\documentclass[a4paper,12p]{amsart}
\usepackage{amsthm,amsmath,amssymb,mathrsfs,url,amsfonts} 
\theoremstyle{definition}
\newtheorem{thm}{Theorem}[section]
\newtheorem{lem}[thm]{Lemma}
\newtheorem{prop}[thm]{Proposition}

\newtheorem{dfn}{Definition}[section]
\newtheorem{rem}{\textit{Remark}}[section]

\newcommand\DN{\newcommand}
\newcommand\DR{\renewcommand}

\DN\ts{\times}
\DN\bra{\langle}
\DN\ket{\rangle}
\DN\bs{\bigskip}
\DN\ms{\medskip}

\DN\lref[1]{Lemma~\ref{#1}}
\DN\tref[1]{Theorem~\ref{#1}}
\DN\pref[1]{Proposition~\ref{#1}}
\DN\sref[1]{Section~\ref{#1}}
\DN\dref[1]{Definition~\ref{#1}}
\DN\rref[1]{Remark~\ref{#1}} 
\DN\corref[1]{Corollary~\ref{#1}}
\DN\eref[1]{Example~\ref{#1}}
\numberwithin{equation}{section}
\newcounter{Const} 

\DN\Ct[1]{\refstepcounter{Const}c_{\theConst}\label{#1}}
\DN\cref[1]{c_{\ref{#1}}}

\DN{\limi}[1]{\lim_{#1\to\infty}} 	
\DN{\limz}[1]{\lim_{#1\to0}}
\DN{\limdz}[1]{\lim_{#1\downarrow 0}}
\DN{\limsupi}[1]{\limsup_{#1\to\infty}}
\DN{\limsupz}[1]{\limsup_{#1\to0}}
\DN{\limsupdz}[1]{\limsup_{#1\downarrow0}}
\DN{\liminfi}[1]{\liminf_{#1\to\infty}}
\DN{\liminfz}[1]{\liminf_{#1\to0}}
\DN{\liminfdz}[1]{\liminf_{#1\downarrow0}}
\DN{\supnor}[1]{\| #1\|_{\infty}}
\DN{\sumii}[1]{\sum_{#1=1}^{\infty}}
\DN{\sumi}[1]{\sum_{#1=0}^{\infty}}
\DN\tolaw{\stackrel{d}{\longrightarrow }}

\DN\Lm{L^{2}(\mu )}
\DN\Lmuk{L^{2}(\muk )}
\DN\Lploc{L_{\mathrm{loc}}^{p}}
\DN\Lqloc{L_{\mathrm{loc}}^{q}}
\DN\Loneloc{L_{\mathrm{loc}}^{1}}
\DN\Ltwoloc{L_{\mathrm{loc}}^{2}}
\DN\Ci{C^{\infty}}
\DN\Czi{C_{0}^{\infty}}

\DN{\exi}{{}^{\exists}}
\DN{\exiu}{{}^{\exists 1}}
\DN{\any}{{}^{\forall}}

\DN\ord{{\mathcal o}}
\DN\Ord{\mathcal{O}}
\DN\PD[2]{\frac{\partial #1 }{\partial #2}}
\DN\half{\frac{1}{2}}
\DN\map[3]{#1:#2 \to #3}

\DN\laweq{\stackrel{d}{=}}

\DN\R{\mathbb{R}}
\DN\Rd{\mathbb{R}^d}
\DN\N{\mathbb{N}}
\DN\Q{\mathbb{Q}}
\DN{\Z}{\mathbb{Z}}
\DN\C{\mathbb{C}}

\DN\0{\As{I1}--\As{I5}}
\DN\1{\lim_{r \to \infty } \limsup_{N \to \infty} \sup_{x \in S_R} }
\DN\2{\limi{s}\limsupi{N}\sup_{x\in\Sq }}
\DN\3{\aaa _{r,s}^{N,\tail }}
\DN\4{\bbba _{r,s}^{N,\tail }}
\DN\5{\aaa _{r,s}^{N,\tail}}
\DN\6{\bbb _{r,s}^{N,\tail}} 
\DN\7{\bNrs - \bNrsp }
\DN\8{\nabla _i \hhh (\mathbf{X}_u^{m } )}

\DN\9{ \int_0^{t} \half 
\aaa (X_u^{i},\Xidu ) \nabla _i \nabla _i 
\hhh (\mathbf{X}_u^{m } ) + \btail ( X_u^{i} ) 
\cdot \nabla _i \hhh ( \mathbf{X}_u^{m } ) du }
\DN\Lp{L^{\phat }( \muonebar )}
\DN\Lplochat{L_{\mathrm{loc}}^{p}( \muone )}
\DN\LNp{L^{\phat }( \muNonebar )}

\DN\nablax{\nabla _{x}}

\DN\hhh{\psi}
\DN\rsNi{_{r,s,\p }^{N,i}}
\DN\mrs{$ m,r,s \in \N $ such that $ r < s $}
\DN\mart{\sum_{i=1}^m \int_0^{\cdot }
\nabla _i \hhh (\mathbf{X}_u^m) \cdot \sigma (X_u^{i},\Xidu )d\hat{B} _u^i }

\DN\Mart{\sum_{i=1}^m \int_0^{\cdot }
\nabla _i \hhh (\mathbf{X}_u^m) \cdot \sigma (X_u^{i},\Xidu )d\hat{B} _u^i }

\DN\MartN{\sum_{i=1}^m \int_0^{\cdot }
\nabla _i \hhh (\mathbf{X}_u^{\nN ,m}) 
\cdot \sigmaN (X_u^{\nN ,i},\XNidu )dB _u^i }

\DN\MartNp{\sum_{i=1}^m \int_0^{\cdot }
\nabla _i \hhh (\mathbf{X}_u^{\nN ,m}) 
\cdot \sigmaN _{\mathsf{p}}(X_u^{\nN ,i},\XNidu )dB _u^i }

\DN\nc{\nabla _i \hhh ( \mathbf{X}_u^{\nN , m } )}
\DN\grad{\mathrm{grad}}
\DN\ot{\otimes}
\DN\Ai{\mathrm{Ai}}

\DN\phat{\hat{p}}
\DN\p{\mathsf{p}} \DN\q{\mathsf{q}}
\DN\erf{\mathrm{Erf}}
\DN\muonebar{\bar{\mu} _r^{[1]}}
\DN\tail{\mathrm{tail}}
\DN\supN{\sup_{\nN \in \N }}
\DN\nN{N}%\DN\nN{\mathcal{N}}
\DN\n{N}%\DN\n{\mathcal{N}}
\DN\sigmaN{\sigma ^{\nN }}
\DN\sN{\sigma ^{\nN }}
\DN\MNi{\mathsf{M} ^{\nN ,i}}
\DN\MN{\mathsf{M}^{\nN }}
\DN\aN{\mathsf{a}^{\nN }} 
\DN\aNi{\mathsf{a}^{\nN ,i}} 
\DN\aNrs{\mathsf{a}^{\nN }}

\DN\bN{\mathsf{b}^{\nN }} 
\DN\bNi{\mathsf{b}^{\nN ,i}}
\DN\bNrs{\mathsf{b}_{r,s}^{\nN }}
\DN\bNrsp{\mathsf{b}_{r,s,\p }^{\nN }}
\DN\bNrsq{\mathsf{b}_{r,s,\q }^{\nN }}
\DN\brsp{\mathsf{b}_{r,s,\p }}
\DN\brsq{\mathsf{b}_{r,s,\q }}
\DN\brs{\mathsf{b}_{r,s }}
\DN\br{\mathsf{b}_{r}}
\DN\bNrstail{\mathsf{b}_{r,s}^{\nN ,\tail }}
\DN\btail{\mathsf{b} ^{\tail }}
\DN\cNrsp{\mathsf{c}_{r,s,\p }^{\nN }}
 \DN\ANirs{\mathsf{A}^{\nN ,i}} 
 \DN\ANi{\mathsf{A}^{\nN ,i}} 
 \DN\ANjrs{\mathsf{A}^{\nN ,j}} 
 \DN\Airs{\mathsf{A}^{i}} 
 \DN\ANrs{\mathsf{A}^{\nN }} 

\DN\Birs{\mathsf{B}_{r,s}^{i}}
\DN\Birsp{\mathsf{B}_{r,s,\p }^{i}}
\DN\Brsm{\mathbf{B}_{r,s}^{m}}
\DN\Brspm{\mathbf{B}_{r,s,\p }^{m}}
\DN\Cirsp{\mathsf{C}_{r,s,\p }^{i}}
\DN\BNi{\mathsf{B}^{\nN ,i}}
\DN\BNirs{\mathsf{B}_{r,s}^{\nN ,i}}
\DN\BNjrs{\mathsf{B}_{r,s}^{\nN ,j}}
\DN\BNirsp{\mathsf{B}_{r,s,\p }^{\nN ,i}}
\DN\BNrs{\mathsf{B}_{r,s}^{\nN }}
\DN\BNrsp{\mathsf{B}_{r,s,\p }^{\nN }}
\DN\BNrspm{\mathbf{B}_{r,s,\p }^{\nN ,m}}
\DN\CNirsp{\mathsf{C}_{r,s,\p }^{\nN ,i}}
\DN\CNjrsp{\mathsf{C}_{r,s,\p }^{\nN ,j}}
\DN\CNrsp{\mathsf{C}_{r,s,\p }^{\nN }}

\DN\labN{\lab _{\nN }} \DN\labNm{\lab _{\nN ,m}} \DN\labm{\lab _{m }}
\DN\Dcp{D_{\mathrm{cp}}}

\DN\RNrsp{\mathsf{R}_{r,s,\p }^{\nN ,m }}
\DN\RN{\mathsf{R}_{r,s }^{\nN }}
\DN\QN{\mathsf{Q}_{r,s,\p }^{\nN }}
\DN\Brsp{\mathsf{B}_{r,s,\p }}

\DR\S{S} 
\DR\SS{\mathsf{S}}
\DN\SN{\S ^{\mathbb{N}}}

\DN\iii{\mathrm{i}}
\DN\rrr{{r}}%[]
\DN\ww{\rr } %[]
\DN\www{\rr (x)}
\DN\rr{{w}}
\DN\rrN{\mathsf{\rr }^{N}}
\DN\rrNs{\rrN _{s}}
\DN\rrbar{\bar{\rr }}

\DN\As[1]{\textbf{(#1)}}

\DN\Ss{\mathit{S}}
\DN\Sk{\S ^{k}}
\DN\SkS{\Sk \times \SSS }
\DN\SSS{\mathsf{S}}
\DN\sss{\mathsf{s}}
\DN\SSSsi{\SSS _{\mathrm{s.i.}}}
\DN\SoneSSS{\S \times \SSS }
\DN\SSSksingle{\SSSsi ^{k}}
\DN\Srm{\SSS _{r}^{m}}
\DN\SSSS{\mathbb{S}}
\DN\SR{\S _R}
\DN\SrSS{\Sr \ts \SS } 
\DN\Sr{\S _r}

\DN\AAA{\mathsf{A}}

\DN\xsss{(x,\sss )}

\DN\muNl{\muN \circ \labN ^{-1}}

\DN\muone{\mu ^{[1]}}
\DN\muN{\mu ^{N}}
\DN\muNcheck{\check{\mu }^{\nN }}
\DN\muNx{\mu ^{N}_{x}}
\DN\muNzero{\mu ^{N}_{0}}
\DN\muNone{\mu ^{N,{[1]}}}
\DN\muk{\mu ^{{[k]}}}
\DN\muNm{\mu ^{{N,[m]}}}
\DN\muNonebar{\bar{\mu }^{N,{[1]}}_{\rrr }}

\DN\mul{\mu \circ \lab ^{-1}}

\DN\bbb{\mathsf{b}}
\DN\aaa{\mathsf{a}}
\DN\xx{\mathbf{x}}
\DN\xxx{\mathsf{x}}
\DN\yy{\mathbf{y}}
\DN\yyy{\mathsf{y}}
\DN\zz{\mathbf{z}}
\DN\zzz{\mathsf{z}}
\DN\xyi{|x-y_i|}

\DN\Sq{\S _{\rrr }}

\DR\d{\mathcal{D}} 
\DN\di{\d _{\circ}}
\DN\dione{C^{\infty}_{0} (\S )\otimes \di }
\DN\dik{\di ^{k}}
\DN\dak{\d ^{\aaa ,k}}

\DN\DDD{\mathbb{D}^{\aaa }}
\DN\DDDk{\mathbb{D}^{\aaa ,k}}

\DN\E{\mathcal{E}}
\DN\Ea{\E ^{\aaa }}
\DN\Eak{\E ^{\aaa ,k}}

\DN\dlog{\mathsf{d}}
\DN\dmu{\dlog ^{\mu }}
\DN\dmuN{\dlog ^{N}}
\DN\dmuone{\dlog ^{\mu ^1}}

\DN\dnuik{C_0^{\infty}(\Sk )\otimes \di }

\DN\ulab{\mathfrak{u} }
\DN\lab{\ell  }
\DN\lpath{\lab _{\mathrm{path}}}
\DN\lkpath{\lab _{k,\mathrm{path}}}
\DN\upath{\ulab _{\mathrm{path}}}

\DN\PP{\mathbf{P}}
\DN\PPs{\mathbf{P}_{\mathbf{s}}}
\DN\PPP{\mathsf{P}}
\DN\PPPs{\mathsf{P} _{\mathsf{s}}}
\DN\Pmu{\mathsf{P}_{\mu }}
\DN\PmuX{\mathsf{P}_{\mu }^{\XXX _t}}
\DN\Pxt{\PPP _{(x,\sss ) }}
\DN\PPk{\PPP ^{k}}
\DN\Pmt{\PPP _{\sss }}
\DN\PPmul{\PP _{\mul }}
\DN\PPmulN{\PP _{\muNl }}
\DN\PmuN{\PPP _{\mu ^N}}

\DN\g{g}
\DN\gN{\g ^{N}}
\DN\gNrs{\gN _{rs}}
\DN\gNs{\gN _{s}}

\DN\ggN{\mathsf{g} ^{N}}
\DN\ggNrs{\ggrs ^{N}}
\DN\ggNst{\ggst ^{N}}
\DN\ggNrt{\ggrt ^{N}}
\DN\ggNs{\ggN _{s}}
\DN\ggNr{\ggN _{r}}
\DN\ggrs{\mathsf{g}_{rs}} 
\DN\ggst{\mathsf{g}_{st}} 
\DN\ggrt{\mathsf{g}_{rt}} 

\DN\ggr{\mathsf{g}_{r}}
\DN\ggtilder{\ggtilde_{r}}
\DN\ggs{\mathsf{g}_{s}}
\DN\ggtildes{\ggtilde_{s}}
\DN\ggtilde{\tilde{\mathsf{g}}}
\DN\ggtilders{\tilde{\mathsf{g}}_{rs}}

\DN\wNr{{\mathsf w}_r^N}

\DN\hh{\mathfrak{h}}
\DN\hsym{h _{\mathrm{sym}}}

\DN\Y{(\mathsf{y}=\sum_{i=1}^{N-1}\delta_{y_i})}

\DN\uu{u}
\DN\uN{\uu ^{N}}
\DN\vN{v^{N}}

\DN\vNsinfty{\int_{\S } \{1- \chi _s (x-y)\} \vN (x,y)dy }
\DN\vNs{\int_{\S } \chi _s (x-y) \vN (x,y)dy }

\DN\vvv{v}
\DN\vsinfty{\int \{1- \chi _s (x-y)\} \vvv (x,y)dy }
\DN\vs{\int_{\S }  \chi _s (x-y) \vvv (x,y)dy }
\DN\rNone{\rho ^{N,\mathrm{1}}}
\DN\rNtwo{\rho ^{N,2}}
\DN\rNnk{\rho ^{N,n+k}}

\DN\F{{\mathcal F}}
\DN\Floc{\dot{{\mathcal F}}_{\mathrm{loc}}}
\DN\Fe{{\mathcal F}_{e}}
\DN\M{{\mathcal M}}
\DN\MA{\mathring{\M }}
\DN\MAloc{\mathring{\M }_{\mathrm{loc}}}
\DN\Nc{{\mathcal N} _c}
\DN\Ncloc{{\mathcal N} _{c,\mathrm{loc}}}
\DN\PCAF{\mathbf{A}_c^{+}}
\DN\Capa{\mathrm{Cap}}

\DN\XX{\mathbf{X}}
\DN\XXX{\mathsf{X}}
\DN\XXtt{\mathbf{X}_{\mathbf{t}}}
\DN\XXXtt{\mathsf{X}_{\mathbf{t}}}
\DN\Xm{\mathbf{X}^{m}}
\DN\XXXtms{\mathsf{X}_{t}^{m*}}
\DN\Xms{\mathbf{X}^{m*}}
\DN\XNm{\mathbf{X}^{N,m}}

\DN\diai{\diamond i}
\DN\diaj{\diamond j}
\DN\Xid{\XXX^{\diai }}
\DN\Xidt{\XXX _{t}^{\diai }}
\DN\Xidu{\XXX _{u}^{\diai }}

\DN\Xjdu{\XXX _{u}^{\diaj }}
\DN\Xjdt{\XXX _{t}^{\diaj }}
\DN\XNidt{\XXX _{t}^{N,\diai }}
\DN\XNidu{\XXX _{u}^{N,\diai }}
\DN\XNidv{\XXX _{v}^{N,\diai }}
\DN\XNjdu{\XXX _{u}^{N,\diaj }}
\DN\XNjdv{\XXX _{v}^{N,\diaj }}
\DN\XNjdz{\XXX _{0}^{N,\diaj }}
\DN\XNjdt{\XXX _{t}^{N,\diaj }}
\DN\BB{\mathbf{B}}

\DN\YY{\mathbf{Y}}
\DN\VV{\mathbf{V}}

\DN\HHH{\mathsf{H}}
\DN\Hone{\HHH^{[1]}}

\DN\WH{C([0,\infty);\HHH )}
\DN\WSN{C([0,\infty);\S ^\N )}
\DN\WTSN{C_T(\S ^\N )}
\DN\WTSS{C_T(\SSS )}
\DN\WTSSsi{C_T(\SSSsi )}
\DN\WTSNz{C_T^{\mathbf{0}}(\S ^\N )}
\DN\Wsol{W_{T,\mathrm{sol}}}
\DN\Wsols{W_{T,\mathrm{sol}}^{\mathbf{s}}}
\DN\Wfixs{W_{T,\mathrm{fix}}^{\mathbf{s}}}

\DN\Tail{\mathrm{Tail}}

\DN\SSSz{\SSS _0}
\DN\SSz{\mathbf{S}_0}

\DN\Psb{\bar{P}_{\mathbf{s}}}
\DN\Psbb{\bar{P}_{\mathbf{s},\mathbf{B}}}
\DN\PBr{P_{\mathrm{Br}}^\infty}

\DN\Tpath{\mathcal{T}_{\mathrm{path}} (\S ^\N)}
\DN\Tpathm{\mathcal{T}_{\mathrm{path}}^m (\S ^\N)}
\DN\Tpone[1]{\mathcal{T}_{\mathrm{path}}^{[1]} (\S ^\N ; #1)}
\DN\Tpcone[1]{\tilde{\mathcal{T}}_{\mathrm{path}}^{[1]} (\S ^\N ; #1)}
\DN\Tpc{\tilde{\mathcal{T}}_{\mathrm{path}} (\S ^\N)}
\DN\TS{\mathcal{T}(\SSS )}
\DN\TSp[1]{\mathcal{T}_{\mathrm{path}}^{\mathbf{#1}}(\SSS )}
\DN\TSone[1]{\mathcal{T}^{[1]}(\SSS ,#1)}
\DN\TSpc{\tilde{\mathcal{T}}_{\mathrm{path}} (\SSS )}
\DN\TSpcone[1]{\tilde{\mathcal{T}}_{\mathrm{path}}^{[1]} (\SSS ; #1)}

\DN\U{\mathcal{U}}
\DN\V{\mathcal{V}}
\DN\CCC{\mathsf{C}}
\DN\Cs[1]{\CCC ^{[\mathsf{#1}]}}

\DN\Fsi{F_{\mathbf{s}}^\infty}
\DN\Fsm{F_{\mathbf{s}}^m}
\DN\Fsii{F_{\mathbf{s}}^{\infty,i}}
\DN\Fsmi{F_{\mathbf{s}}^{m,i}}
\DN\Fsbi{F_{\mathbf{s},\mathbf{B}}^\infty}
\DN\Fsbm{F_{\mathbf{s},\mathbf{B}}^m}

\DN\mrT{\mathsf{m}_{r,T}}
\DN\MrTi{\mathsf{M}_{r,T}^i}

\DN\sigmaXm{\sigma _{\XXX }^m}
\DN\bXm{b _{\XXX }^m}

\DN\muTail{\mu _{\mathrm{Tail}}^{\mathsf{a}}}

\DN\ZNi{Z^{N,i}}

\DN\KtwoN{K _{2}^N}
\DN\StN{S _{\theta } ^{N}}
\DN\muNt{\mu _\theta ^N}

\DN\Srx{S _{r,\infty}^x}
\DN\SrxN{S _{r,\infty}^{x,N}}

\DN\osc[1]{\psi _{#1}}

\DN\rN{\sqrt{N}}
\DN\orN{\frac{1}{\sqrt{N}}}

\DN\Nyt{\frac{1}{x-Ny+N\theta } }
\DN\Nzt{\frac{1}{x-Nz+N\theta } }

\begin{document}

\pagestyle{myheadings}%\pagestyle{plain}
\markboth{Finite-particle approximations for interacting Brownian particles with logarithmic potentials}{Finite-particle approximations for interacting Brownian particles with logarithmic potentials}

\title
{Finite-particle approximations for interacting Brownian particles with logarithmic potentials }
\author{Yosuke Kawamoto,\quad Hirofumi Osada}

\maketitle

%\subjclass[2010]{Primary 60B20; Secondary 60H10}
\vskip 0.5cm

\begin{abstract}
We prove the convergence of $ \nN $-particle systems of Brownian particles with logarithmic interaction potentials onto a system described by the infinite-dimensional stochastic differential equation (ISDE). 
For this proof we present two general theorems on the finite-particle approximations of interacting Brownian motions. 
In the first general theorem, we present a sufficient condition for a kind of tightness of solutions of stochastic differential equations (SDE) describing finite-particle systems, and prove that the limit points solve the corresponding ISDE. This implies, if in addition the limit ISDE enjoy a uniqueness of  solutions, then the full sequence converges. 
We treat non-reversible case in the first main theorem. 
In the second general theorem, we restrict to the case of reversible particle systems and simplify the sufficient condition. We deduce the second theorem from the first.  
We apply the second general theorem to Airy$ _{\beta }$ interacting Brownian motion with $ \beta = 1,2,4$, and the Ginibre interacting Brownian motion. 
The former appears in the soft-edge limit of Gaussian (orthogonal/unitary/symplectic) ensembles in one spatial dimension, and the latter in the bulk limit of Ginibre ensemble in two spatial dimensions, corresponding to a quantum statistical system for which the eigen-value spectra belong to non-Hermitian Gaussian random matrices. The passage from the finite-particle stochastic differential equation (SDE) to the limit ISDE is a sensitive problem because the logarithmic potentials are long range and unbounded at infinity. Indeed, the limit ISDEs are not easily detectable from those of finite dimensions. 
Our general theorems can be applied straightforwardly to the grand canonical Gibbs measures with Ruelle-class potentials such as Lennard-Jones 6-12 potentials and and Riesz potentials. 
\end{abstract}

\maketitle

\section{Introduction}\label{s:1}

Interacting Brownian motion in infinite dimensions 
is prototypical of diffusion processes of infinitely many particle systems, 
initiated by Lang \cite{lang.1,lang.2}, followed by Fritz \cite{Fr}, Tanemura \cite{tane.2}, and others. Typically, interacting Brownian motion $ \mathbf{X}=(X^i)_{i\in\N }$ with Ruelle-class (translation invariant) interaction $ \Psi $ and inverse temperature $ \beta \ge 0 $ is given by 
\begin{align}\label{:10a}& \quad 
dX_t^i = dB_t^i - \frac{\beta }{2} 
\sum_{j;j\not=i} ^{\infty}\nabla \Psi (X_t^i-X_t^j ) dt 
\quad (i\in\N )
.\end{align}
Here an interaction $ \Psi $ is called Ruelle-class if $ \Psi $ is super stable in the sense of Ruelle, and integrable at infinity \cite{ruelle.2}.

The system $ \mathbf{X}$ is a diffusion process with state space 
$ \mathbf{S}_0 \subset (\Rd )^{\N }$, and has no natural invariant measures. Indeed, such a measure $ \check{\mu }$, if exists, is informally given by 
\begin{align}\label{:10b}&
\check{\mu } = \frac{1}{\mathcal{Z}} 
 e^{- \beta \sum_{(i,j);\, i<j} ^{\infty}\Psi (x_i-x_j)} 
 \prod_{k\in\N } dx_k 
,\end{align}
which cannot be justified as it is because of the presence of an infinite product of Lebesgue measures. 
To rigorize the expression \eqref{:10b}, the Dobrushin--Lanford--Ruelle (DLR) framework introduces the notion of a Gibbs measure. A point process $ \mu $ is called a $ \Psi $-canonical Gibbs measure if it satisfies the DLR equation: for each $ m \in \N $ and $ \mu $-a.s.\! $ \xi =\sum_i \delta_{\xi _i}$ 
\begin{align}\label{:10c}&
\mu _{r,\xi }^m (d\mathsf{s} ) 
= \frac{1}{\mathcal{Z}_{r,\xi }^m } 
e^{- \beta \{\sum_{i<j ,\, s_i , s_j \in \Sr }^m \Psi (s_i-s_j) + 
\sum_{ s_i \in \Sr, \xi _j \in \Sr ^c }^m \Psi (s_i-\xi _j) \} 
} \prod_{k=1}^m ds_k 
,\end{align}
where $ \mathsf{s} = \sum_i \delta_{s_i}$, $ \Sr = \{ |x| \le r \} $, 
$ \pi_r (\mathsf{s}) = \mathsf{s} (\cdot \cap \Sr )$, and 
$ \xi $ is the outer condition. 
Furthermore, $ \mu _{r,\xi }^m$ denotes the regular conditional probability: 
\begin{align*}&
\mu _{r,\xi}^m (d\mathsf{s} ) = 
\mu (\pi_r (\mathsf{s}) \in d\mathsf{s} | \, \mathsf{s} (\Sr ) = m ,\, \pi_r^c (\mathsf{s})=\pi_r^c (\xi )) 
.\end{align*}
Then $ \mu $ is a reversible measure of 
the delabeled dynamics $ \mathsf{X}$ such that 
$ \mathsf{X}_t = \sum_{i\in \N } \delta _{X_t^i}$. 

If the number of particles is finite, $ \nN $ say, then SDE \eqref{:10a} becomes 
\begin{align}\label{:10e}&
dX_t^{\nN ,\, i } = dB_t^i 
- \frac{\beta }{2} \{ \nabla \Phi ^{\nN }(X_t^{\nN ,\, i })
+ 
\sum_{j;j\not=i} ^{\nN }\nabla \Psi (X_t^{\nN ,\, i }-X_t^{\nN ,\, j }) \} dt 
\quad (1 \le i \le \nN ) 
,\end{align}
where $ \Phi ^{\nN }$ is a confining free potential vanishing zero as $ \nN $ 
goes to infinity. 
The associated labeled measure is then given by 
\begin{align}\label{:10f}&
\muNcheck = \frac{1}{\mathcal{Z}} 
 e^{- \beta \{ \sum_{i=1}^{\nN } \Phi ^{\nN } (x_i) 
+ \sum_{(i,j);\, i < j} ^{\nN }\Psi (x_i-x_j) \} 
} 
 \prod_{k=1 }^{\nN } dx_k 
.\end{align}
The relation between \eqref{:10e} and \eqref{:10f} is as follows. 
We first consider the diffusion process associated with the Dirichlet form 
with domain $ \mathcal{D}^{\muNcheck } $ on $ L^2 ((\Rd )^{\nN } ,\muNcheck )$, 
called the distorted Brownian motion, such that 
$$ \mathcal{E}^{\muNcheck } (f,g) = \int_{(\Rd )^{\nN } } \half \sum_{i=1}^{\nN } 
\nabla _i f \cdot \nabla _i g 
\, \muNcheck (d\mathbf{x}_{\nN })
,$$
where $ \nabla _i = (\PD{}{x_{ij}})_{j=1}^d$, 
$ \mathbf{x}_{\nN } =(x_1,\ldots,x_{\nN } ) \in (\Rd )^{\nN } $, 
 and $ \cdot $ denotes the inner product in $ \Rd $. 
The generator $ - L^{\muNcheck } $ of $ \mathcal{E}^{\muNcheck }$ is then given by 
$$ \mathcal{E}^{\muNcheck } (f,g) = 
 ( - L^{\muNcheck } f, g )_{L^2((\Rd )^{\nN } , \muNcheck )}
.$$
Integration by parts yields the representation of 
the generator of the diffusion process such that 
\begin{align*}&
L^{\muNcheck } =  \half \Delta - \frac{\beta }{2} 
\sum_{i=1}^{\nN } \{ 
\nabla {\Phi ^{\nN }} (x_i)
+ \sum_{j;\, j\not=i } ^{\nN }\nabla \Psi (x_i-x_j) 
\}\cdot \nabla _i 
,\end{align*}
which together with It$ \hat{\mathrm{o}}$ formula yields SDE \eqref{:10e}. 

For a finite or infinite sequence $ \mathbf{x}=( x_i ) $, we set 
$ \ulab (\mathbf{x}) = \sum_i \delta_{x_i}$ and call $ \mathfrak{u}$ 
a delabeling map. 
For a point process $ \mu $, we say a measurable map $ \lab = \lab (\mathsf{s})$ 
defined for $ \mu $-a.s.\! $ \mathsf{s}$ 
with value $ \S ^\infty \cup \{\bigcup _{\nN =1}^{\infty}\S ^{\nN } \} $ 
is called a label with respect to $ \mu $ if 
$\ulab \circ \lab (\mathsf{s}) = \mathsf{s}$. 
Let $ \labN $ be a label with respect to $ \muN $. 
We denote by $ \lab _m $ and $ \labNm $ the first $ m $-components of these labels, respectively. We take $ \Phi ^{\nN } $ such that the associated point process $ \muN = \muNcheck \circ \ulab ^{-1}$ converges weakly to $ \mu $:
\begin{align}\label{:10g}&
\limi{\nN } \muN = \mu \quad \text{ weakly} 
.\end{align}
The associated delabeling 
$ \mathsf{X}^{\nN } = \sum_{i=1}^{\nN } 
\delta _{X^{\nN ,\, i}}$ is reversible with respect to $ \muN $. The labeled process $ \mathbf{X}=(X^i)$ and $ \mathbf{X}^{\nN } = (X^{\nN ,\, i})$ 
can be recovered from $ \mathsf{X}$ and $ \mathsf{X}^{\nN }$ by taking suitable initial labels $ \lab $ and $ \labN $, respectively. 
Choosing the labels in such a way that for each $ m \in \N $
\begin{align}\label{:10h}&
\limi{\nN } \muN \circ \labNm ^{-1 } = 
\mu \circ \labm ^{-1}\quad \text{ weakly}
,\end{align}
we have the convergence of labeled dynamics $ \mathbf{X}^{\nN }$ to $ \mathbf{X}$ such that for each $ m $
\begin{align}\label{:10i}&
\limi{\nN } (X^{\nN ,1},\ldots,X^{\nN ,m }) = 
 (X^{1},\ldots,X^{m }) 
\quad \text{ in law in } C([0,\infty ); (\Rd )^m )
.\end{align}
We expect this convergence because of 
the absolute convergence of the drift terms in \eqref{:10a} and 
energy in the DLR equation \eqref{:10c} for well-behaved initial distributions 
although it still requires some work to justify this rigorously even if $ \Psi \in C_0^3(\Rd )$ \cite{lang.1}. 

If we take logarithmic functions as interaction potentials, 
then the situation changes drastically. Consider the soft-edge scaling limit of 
Gaussian (orthogonal/unitary/symplectic) ensembles. 
Then the $ \nN $-labeled density is given by 
\begin{align}\label{:10k}&
\check{\mu } _{\mathrm{Airy}, \beta } ^{\nN }(d\mathbf{x}_{\nN })= 
\frac{1}{Z}
\{ \prod_{i<j}^{\nN }|x_i-x_j|^\beta \} 
\exp\bigg\{-\frac{\beta}{4}\sum_{k=1}^{\nN } |2\sqrt{\nN }+\nN ^{-1/6}x_k|^2 \bigg\}
d\mathbf{x}_{\nN } 
\end{align}
and the associated $ \nN $-particle dynamics described by SDE
\begin{align}\label{:10l}&
dX_t^{\n ,i} = dB_t^i + 
\frac{\beta }{2} \sum_{j=1,\, j\not= i}^{\n } 
\frac{1}{X_t^{\n ,i} - X_t^{\n ,j} } dt 
- \frac{\beta }{2 } 
\{ \n ^{1/3} + \frac{1}{2\n ^{1/3}}X_t^{\n ,i} \}dt 
.\end{align}
The correspondence between \eqref{:10k} and \eqref{:10l} is transparent and 
 same as above. Indeed, we first consider distorted Brownian motion 
(Dirichlet spaces with $ \check{\mu } _{\mathrm{Airy}, \beta } ^{\nN }$ 
as a common time change and energy measure), 
then we obtain the generator of the associated diffusion process by integration by parts. SDE \eqref{:10l} thus follows from the generator immediately. 

It is known that the thermodynamic limit $\mu _{\mathrm{Airy},\beta }$
of the associated point process $ \mu _{\mathrm{Airy}, \beta } ^{\nN }$ 
exists for each $ \beta >0$ \cite{rrv.airy}. 
Its $ m $-point correlation function is explicitly given as a determinant of certain kernels if $ \beta = 1,2,4$ \cite{AGZ,mehta}. Indeed, if $ \beta = 2 $, then the $ m $-point correlation function of 
the limit point process $ \mu _{\mathrm{Airy},2}$ is 
\begin{align} \notag &%\label{:11e}&
\rho _{\Ai ,2}^{m} (\mathbf{x}_{m}) = 
\det [K_{\Ai , 2}(x_i,x_j) ]_{i,j=1}^{m} 
,\end{align}
where $ K_{{\Ai },2}$ is the continuous kernel such that, for $ x \not= y $, 
\begin{align} & \notag %\label{:11g}&
K_{{\Ai },2}(x,y) = 
\frac{{\Ai }(x) {\Ai }'(y)-{\Ai }'(x) {\Ai }(y)}{x-y}
.\end{align}
We set here ${\Ai }'(x)=d {\Ai }(x)/dx$ and 
denote by ${\Ai }(\cdot)$ the Airy function given by 
\begin{align}& \notag %\label{:11h}&
{\Ai }(z) = \frac{1}{2 \pi} \int_{\R} dk \,
e^{i(z k+k^3/3)},
\quad z\in\R 
.\end{align}
For $ \beta = 1,4 $ similar expressions in terms of the quaternion determinant 
are known \cite{AGZ,mehta}. 

From the convergence of equilibrium states, we may expect the convergence of solutions of SDEs \eqref{:10l}. 
The divergence of the coefficients in \eqref{:10l} and the very long-range nature of the logarithmic interaction however prove to be problematic. 
Even an informal representation of the limit coefficients is nontrivial but 
has been obtained in \cite{o-t.airy}. Indeed, the limit ISDEs are given by 
\begin{align} \label{:10H} &
dX_t^i=dB_t^i+
\frac{\beta }{2}
\limi{r} \{ \sum_{|X_t^j|<r, j\neq i} \frac{1}{X_t^i-X_t^j} -\int_{|x|<r}\frac{\varrho (x)}{-x}\,dx\}dt \quad (i\in\N)
.\end{align}
Here $ \varrho (x) = 1_{(-\infty , 0)} (x) \sqrt{-x}$, which is the shifted and rescaled semicircle function at the right edge. 

As an application of our main theorem (\tref{l:22}), 
we prove the convergence \eqref{:10i} of solutions from \eqref{:10l} to \eqref{:10H} for 
 $ \{ \mu _{\mathrm{Airy}, \beta } ^{\nN } \} $ with $ \beta = 2 $. 
We also prove that the limit points of solutions of \eqref{:10l} 
satisfy ISDE \eqref{:10H} with $ \beta = 1,2,4 $. 

For general $ \beta \not=1,2,4$, the existence and uniqueness of solutions of \eqref{:10H} is still an open problem. 
Indeed, the proof in \cite{o-t.airy} relies on a general theory developed in 
\cite{o.tp,o.isde,o.rm,o.rm2,o-t.tail}, which reduces the problem to the quasi-Gibbs property and the existence of the logarithmic derivative of the equilibrium 
state. These key properties are proved only for $ \beta = 1,2,4 $ at present. 
We refer to \cite{o.rm,o.rm2} for the definition of the quasi-Gibbs property and 
\dref{d:21} for the logarithmic derivative.

Another typical example is the  Ginibre interacting Brownian motion, which is an infinite-particle system in $ \R ^2$ (naturally regarded as $ \mathbb{C}$), whose equilibrium state is  the Ginibre point process $ \mu _{\mathrm{gin}} $. 
The $ m $-point correlation function $ \rho_{\mathrm{gin}}^m $ with respect to 
Gaussian measure $ (1/\pi ) e^{-|x|^2}dx $ on $ \mathbb{C}$ is then given by 
\begin{align}& \notag %\label{:10p}&
 \rho_{\mathrm{gin}}^m (\mathbf{x}_m)=\det [e^{x_i\bar{x}_j}]_{i,j=1}^m 
.\end{align}
The Ginibre point process $ \mu _{\mathrm{gin}} $ 
is the thermodynamic limit of $ \nN $-particle point process 
$ \mu _{\mathrm{gin}}^{\nN } $ whose labeled measure is given by 
\begin{align}& \notag %\label{:10q}&
\check{\mu }_{\mathrm{gin}}^{\nN } (d\mathbf{x}_{\nN }) = 
\frac{1}{\mathcal{Z} } \prod_{i<j}^{\nN }|x_i-x_j|^2 
e^{-\sum_{i=1}^{\nN }|x_i|^2} d\mathbf{x}_{\nN }
%\prod_{i=1}^{\nN } dx_i
.\end{align}
The associated $ \nN $-particle SDE is then given by 
\begin{align}\label{:10r} 
dX_t^{N,i} = \, & dB_t^i-X_t^{N,i}dt +\sum_{j=1, j\neq i}^{\nN } \frac{X_t^{N,i}-X_t^{N,j}}{|X_t^{N,i}-X_t^{N,j}|^2}dt \quad (1\le i\le \nN )
.\end{align}
We shall prove that the limit ISDEs are 
\begin{align}\label{:10s}
dX_t^i= \, & 
dB_t^i+\limi{r} \sum_{|X_t^i-X_t^j|<r, j\neq i} \frac{X_t^i-X_t^j}{|X_t^i-X_t^j|^2}dt \quad (i\in\N)
\\\intertext{
{\em and }}
\label{:10t}
dX_t^i= \, & 
dB_t^i - X_t^idt + 
\limi{r} \sum_{|X_t^j|<r, j\neq i} \frac{X_t^i-X_t^j}{|X_t^i-X_t^j|^2}dt \quad (i\in\N)
.\end{align}
In \cite{o.isde, o-t.tail}, it is proved that these ISDEs have the same pathwise unique strong solution for $ \mu _{\mathrm{gin}}\circ \lab^{-1} $-a.s. $ \mathbf{s}$, 
where $ \lab $ is a label and $ \mathbf{s}$ is an initial point. 
As an example of applications of our second main theorem (\tref{l:22}), 
we prove the convergence of 
solutions of \eqref{:10r} to those of \eqref{:10s} and \eqref{:10t}. 
This example indicates again the sensitivity of the representation of the limit ISDE. Such varieties of the limit ISDEs are a result of the long-range nature of the logarithmic potential.  

The main purpose of the present paper is to develop a general theory for finite-particle convergence applicable to logarithmic potentials, and in particular, the Airy and Ginibre point processes. Our theory is also applicable to essentially all Gibbs measures with Ruelle-class potentials such as the Lennard-Jones 6-12 potential and Riesz potentials. 

In the first main theorem (\tref{l:21}), we present a sufficient condition for a kind of tightness of solutions of stochastic differential equations (SDE) describing finite-particle systems, and prove that the limit points solve the corresponding ISDE. This implies, if in addition the limit ISDE enjoy uniqueness of  solutions, then the full sequence converges. 
We treat non-reversible case in the first main theorem. 

In the second main theorem (\tref{l:22}), we restrict to the case of reversible particle systems and simplify the sufficient condition. Because of reversibility, the sufficient condition is reduced to the convergence of logarithmic derivative of $ \mu ^{\nN }$ with marginal assumptions. 
We shall deduce \tref{l:22} from \tref{l:21} and 
apply \tref{l:22} to all examples in the present paper. 

If $ \Psi (x) = -\log |x|$, $ \beta = 2 $ and $ d=1$, there exists an algebraic method to construct the associated stochastic processes \cite{j.02,j.03,KT07b,KT11}, and to prove the convergence of finite-particle systems \cite{o-t.sm,o-t.core}. This method requires that interaction $ \Psi $ is the logarithmic function with $ \beta = 2$ and depends crucially on an explicit calculation of space-time determinantal kernels. It is thus not applicable to $ \beta \not=2$ even if $ d=1$. 

As for Sine$_{\beta } $ point processes, Tsai proved the convergence of finite-particle systems for all $ \beta \ge 1$ \cite{tsai.14}. His method relies on a coupling method based on monotonicity of SDEs, which is very 
specific to this model. 

The organization of the paper is as follows: 
In \sref{s:2}, we state the main theorems (\tref{l:21} and \tref{l:22}). 
In \sref{s:3}, we prove \tref{l:21}. 
In \sref{s:4}, we prove \tref{l:22} using \tref{l:21}. 
In \sref{s:5}, we present examples. 

\section{Set up and the main theorems}\label{s:2}

\subsection{Configuration spaces and Campbell measures}\label{s:21}
Let $ \S $ be a closed set in $ \Rd $ whose interior 
$ \S _{\mathrm{int}}$ is a connected open set 
satisfying $ \overline{\S _{\mathrm{int}}} = \S $ and 
the boundary $ \partial \S $ having Lebesgue measure zero. 
A configuration 
$ \mathsf{s} = \sum_i \delta _{s_i}$ on $ \S $ 
is a Radon measure on $ \S $ consisting of delta masses. 
We set $ \Sr = \{ s \in \S \, ;\, |s| \le r \} $. 
Let $ \SS $ be the set consisting of all configurations of $ \S $. 
By definition, $ \SS $ is given by 
\begin{align} & \notag %& \label{:20a}
\SS = \{ \mathsf{s} = \sum_{i} \delta _{s_i}\, ;\, 
 \text{ $\mathsf{s} ( \Sr ) < \infty $ for each $ r\in\N $} \} 
.\end{align}
By convention, we regard the zero measure as an element of $ \SS $. 
We endow $ \SS $ with the vague topology, which makes 
$ \SS $ a Polish space. 
$ \SS $ is called the configuration space over $ \S $ and a probability measure $ \mu $ on $ (\SS , \mathcal{B}(\SS ) )$ 
is called a point process on $ \S $. 

A symmetric and locally integrable function 
$ \map{\rho ^n }{\S ^n}{[0,\infty ) } $ is called 
the $ n $-point correlation function of a point process $ \mu $ 
on $ \S $ with respect to the Lebesgue measure if $ \rho ^n $ satisfies 
\begin{align} & \notag %\label{:20H}
\int_{A_1^{k_1}\ts \cdots \ts A_m^{k_m}} 
\rho ^n (x_1,\ldots,x_n) dx_1\cdots dx_n 
 = \int _{\SS } \prod _{i = 1}^{m} 
\frac{\mathsf{s} (A_i) ! }
{(\mathsf{s} (A_i) - k_i )!} d\mu 
 \end{align}
for any sequence of disjoint bounded measurable sets 
$ A_1,\ldots,A_m \in \mathcal{B}(\S ) $ and a sequence of natural numbers 
$ k_1,\ldots,k_m $ satisfying $ k_1+\cdots + k_m = n $. 
When $ \mathsf{s} (A_i) - k_i < 0$, according to our interpretation, 
${\mathsf{s} (A_i) ! }/{(\mathsf{s} (A_i) - k_i )!} = 0$ by convention. 
Hereafter, we always consider correlation functions with respect to Lebesgue measures.

A point process $ \mu _{x}$ is called the reduced Palm measure of $ \mu $ 
conditioned at $ x \in \S $ if $ \mu _{x}$ is the regular conditional probability defined as 
\begin{align} & \notag %\label{:20j}
 \mu _{x} = \mu (\cdot - \delta_x | \mathsf{s} (\{ x \} ) \ge 1 )
.\end{align} 
A Radon measure $ \muone $ on $ \S \times \SS $ 
is called the 1-Campbell measure of $ \mu $ if $ \muone $ is given by 
\begin{align}\label{:21h}&
 \muone (dx d\mathsf{s}) = \rho ^1 (x) \mu _{x} (d\mathsf{s}) dx 
.\end{align}

\subsection{Finite-particle approximations (general case)}\label{s:22}

Let $ \{\muN \} $ be a sequence of point processes on $\S $ 
such that $ \muN (\{ \mathsf{s}(\S ) = \nN \} ) = 1 $. We assume: \\
\As{H1} Each $ \muN $ has a correlation function $ \{\rho ^{N,n}\} $ 
satisfying for each $r \in\N$ 
\begin{align}\label{:20f} & 
\limi{N} \rho ^{N,n} (\mathbf{x})= 
\rho ^{n} (\mathbf{x}) \quad \text{ uniformly on $\Sr ^{n}$ for all $n\in\N$} 
,\\ \label{:20g}& 
\sup_{N\in\N } \sup_{\mathbf{x} \in \Sr ^{n}} \rho ^{N,n} (\mathbf{x}) \le 
\cref{;40b} ^{n} n ^{\cref{;40c}n} 
,\end{align}
where $ 0 < \Ct{;40b}(r) < \infty $ and 
$ 0 < \Ct{;40c}(r)< 1 $ are constants independent of $ n \in \N $. 

It is known that \eqref{:20f} and \eqref{:20g} imply weak convergence 
\eqref{:10g} \cite[Lemma A.1]{o.rm}. As in \sref{s:1}, let $ \lab $ and $ \labN $ 
be  labels of $ \mu $ and $ \muN $, respectively. We assume: 

\medskip
\noindent
\As{H2} For each $m\in\N$, \eqref{:10h} holds. That is, 
 \begin{align}\tag{\ref{:10h}}& \quad 
 \limi{N}\mu^{\nN } \circ \labNm ^{-1} =\mu \circ \labm ^{-1}
 \quad \text{ weakly in $ \S ^m $}
 .\end{align}

\medskip 

We shall later take $ \mu^{\nN } \circ \labN ^{-1} $ 
as an initial distribution of a labeled finite-particle system. 
Hence \As{H2} means convergence of the initial distribution of the labeled dynamics. 
There exist infinitely many different labels $\lab $, and 
we choose a label such that the initial distribution of the labeled dynamics converges. \As{H2} will be used in \tref{l:22} and \tref{l:21}. 

\medskip 
For $ \mathbf{X}=(X^i)_{i=1}^{\infty}$ and 
$ \mathbf{X}^{\nN }=(X^{\nN ,i})_{i=1}^{\nN }$, we set 
\begin{align} & \notag %\label{:20l}&
 \Xidt = \sum_{j\not=i}^{\infty} \delta_{X_t^{j}}, \quad \text{ and }\quad 
 \mathsf{X}_t^{N,\diai }= \sum_{j\not=i}^{ N } \delta_{X_t^{N,j}} 
,\end{align}
where $X_t^{N,\diai }$ denotes the zero measure for $ \nN = 1 $. 
Let $ \map{\sN ,\sigma }{\S \ts \SSS }{\R ^{d^2}}$ and 
$ \map{\bbb ^{\nN } ,\bbb }{\S \ts \SSS }{\Rd }$ be measurable functions. 
We introduce the finite-dimensional SDE 
of $ \mathbf{X}^{\nN } =(X^{\nN ,i})_{i=1}^{\nN } $ 
with these coefficients such that for $ 1\le i\le N $ 
%$ \nN $-particles . 
\begin{align}\label{:20m}
dX_t^{N,i} &= 
\sN (X_t^{N,i},\XNidt )dB_t^i + 
\bbb ^{\nN }(X_t^{N,i},\XNidt )dt 
\\\label{:20n}
\XX _0^{\nN } & = \mathbf{s} 
%\XX _0^{\nN } & \stackrel{d}{=} \muN \circ \labN ^{-1}
.\end{align}
We assume: 

\medskip

%\DN\emptyset{\psi }
\noindent \As{H3} 
 SDE \eqref{:20m} and \eqref{:20n} 
has a unique solution for $ \muN \circ \labN ^{-1}$-a.s.\! $ \mathbf{s}$ 
for each $ \nN $: this solution does not explode. 
Furthermore, when $ \partial \S $ is non-void, particles never hit the boundary.

\bs

We set $ \aN = \sN {}^{t}\sN $ and assume: 

\medskip 
\noindent 
\As{H4} 
$\sN $ are bounded and continuous on $ \S \ts \SS $, 
and converge uniformly to $\sigma $ on $ \SrSS $ for each $ r \in \N $. 
Furthermore, $\aN $ are uniformly elliptic on $ \Sr \ts \SS $ for each $ r \in \N $ and $ \nablax  \aN $ are uniformly bounded on $ \S \ts \SS $.

\medskip 

From \As{H4} we see that 
$ \aN $ converge uniformly to $ \aaa := \sigma {}^t\sigma $
 on each compact set $ \SrSS $, and that 
$ \aN $ and $\aaa $ are bounded and continuous on $ \S \ts \SS $. 
There thus 
exists a positive constant $\Ct{;3} $ %depending on $ r $ 
such that 
\begin{align}\label{:20o}& 
||\aaa ||_{\S \ts \SS } ,\ 
||\nabla_x \aaa ||_{\S \ts \SS } , \ 
\sup_{\nN \in\N } ||\aN ||_{\S \ts \SS } ,\ 
\sup_{\nN \in\N } ||\nablax  \aN ||_{\S \ts \SS } \le \cref{;3} 
.\end{align}
Here $ \| \cdot \|_{\S \ts \SS } $ 
denotes the uniform norm on $ \S \ts \SS $. 
Furthermore, we see that $ \aaa $ is uniformly elliptic on each $ \SrSS $. 
From these, we expect that SDEs \eqref{:20m} 
have a sub-sequential limit. 
\begin{align} \notag %\label{:20q}
\limi{\nN }\{X_t^{N,i} - X_0^{N,i}\} &= 
\limi{\nN }\int_0^t \sN (X_t^{N,i},\XNidt )dB_u^i + 
\limi{\nN } \int_0^t \bbb ^{\nN }(X_t^{N,i},\XNidt )du 
\\ \notag 
&= 
\int_0^t \sigma (X_t^{N,i},\XNidt )dB_u^i + 
\limi{\nN } \int_0^t \bbb ^{\nN }(X_t^{N,i},\XNidt )du 
.\end{align}
To identify the second term on the right-hand side and to justify the convergence, we make further assumptions. As the examples in \sref{s:1} suggest, the identification of the limit is a sensitive problem, which is at the heart of the present paper. 

\bs

We set the maximal module variable $ \overline{\mathbf{X}}^{N,m}$ 
of the first $ m $-particles by 
\begin{align}&\notag %\label{:40v} & 
\overline{\mathbf{X}}^{N,m}= \max_{i=1}^m 
 \sup_{t\in [0,T] }|X_t^{N,i}| 
.\end{align}
and by $ \mathcal{L}_r^{N}$ the maximal label with which the particle 
intersects $ \Sr $; that is, 
\begin{align}&\notag %\label{:40q}&
\mathcal{L}_r^{N} = 
\max \{ i \in \N \cup \{ \infty \} \, ;\, |X_t^{N,i}| \le r 
\text{ for some } 0\le t \le T \} 
.\end{align}
We assume the following. 
\\
\As{I1} For each $ m \in \N $ 
\begin{align}\label{:40x}&
\limi{a} \liminfi{\nN }P ^{\muNl } 
(\overline{\mathbf{X}}^{N,m} \le a ) = 1 
% ,\\\label{:40y}& 
% \limi{a} \sup_{\nN \in \N }P ^{\muNl } 
% ( \mathbf{L}_{r}^{\nN } \ge a ) = 0 
\end{align}
and there exists a constant $\Ct{;41a}=\cref{;41a}(m,a)$ such that 
for $ 0 \le t , u \le T $
\begin{align} \label{:40z} & 
\supN \sum_{i=1}^m 
\mathrm{E} ^{\muNl }
[|X _t^{N,i} - X _u^{N,i}|^{4 };\overline{\mathbf{X}}^{N,m} \le a ] 
 \le \cref{;41a} |t-u|^{2} 
.\end{align}
Furthermore, for each $ r \in \N $
\begin{align}\label{:40p}&
\limi{L} \liminfi{\nN }P ^{\muNl } 
(\mathcal{L}_r^{N} \le L ) = 1 
.\end{align}

Let $ \muNone $ be the one-Campbell measure of $ \muN $ 
defined as \eqref{:21h}. Set $ \Ct{;43}(\rrr , N)= \muNone ( \Sr \ts \SSS )$. 
Then by \eqref{:20g} $\sup_{N}\cref{;43}(\rrr ,N)<\infty $ 
for each $\rrr \in \N $. Without loss of generality, 
we can assume that $\cref{;43} > 0 $ for all $\rrr , N $. 
Let $ \muNone _r = \muNone (\cdot \cap \{\Sr \ts \SSS \}) $. 
Let $ \muNonebar $ be the probability measure defined as 
$\muNonebar (\cdot )= 
\muNone (\cdot \cap \{\Sr \ts \SSS \})/\cref{;43}$. 
Let $ \varpi _{r,s}$ be a map from $ \Sr \ts \SS $ to itself such that 
$ \varpi _{r,s} (x,\mathsf{s}) = (x,\sum_{|x-s_i|<s} \delta_{s_i})$, where 
$ \mathsf{s}= \sum_i \delta_{s_i}$. 
Let $ \mathcal{F}_{\rrr ,s} = \sigma [\varpi _{r,s}]$ 
be the sub-$ \sigma $-field of $ \mathcal{B}(\Sr \ts \mathcal{\SS }) $ 
generated by $ \varpi _{r,s}$. Because 
$ \Sr $ is a subset of $ \S $, we can and do regard 
$ \mathcal{F}_{\rrr ,s} $ as a $ \sigma $-field on $ \S \ts \SS $, 
which is trivial outside $ \Sr \ts \SS $. 

We set a tail-truncated coefficient $ \bNrs $ of $ \bN $ and their tail parts $ \6 $ by 
\begin{align}
\label{:40a}&
 \bbb _{r,s}^{\nN } = 
 \mathrm{E}^{\muNonebar }
[ \bbb ^{N}|\mathcal{F}_{\rrr ,s} ] ,\quad 
\bN = \bbb _{r,s}^{\nN } + \6 
.\end{align}
We can and do take a version of $ \bNrs $ such that 
\begin{align}%%&\notag %
\label{:40b}& 
\bNrs (x,\mathsf{y}) = 0 \quad \text{ for } x \not\in \Sr 
,\\ 
\label{:40c}& 
 \bNrs (x,\mathsf{y}) = \bbb _{r+1,s}^{N} (x,\mathsf{y}) 
\quad \text{ for } x \in \S _{r} 
.\end{align}

We next introduce a cut-off coefficient $ \bNrsp $ of $ \bNrs $. 
Let $ \bNrsp $ be a continuous and $ \mathcal{F}_{\rrr ,s}$-measurable function 
on $ \S \ts \SS $ such that 
\begin{align} \label{:40ff}
 \bNrsp (x,\mathsf{y}) &=0 
\quad \text{ for } x \not\in \Sr 
\\
\label{:40f}
 \bNrsp (x,\mathsf{y}) & = \bbb _{r+1,s,\p }^{N} (x,\mathsf{y}) ,\quad 
\quad \text{ for } x \in \S _{r-1} 
\end{align}
and that, for $  (\S \ts \SS )_{r,\p } = \{ (x,\mathsf{y} ) \in \S _{r} \ts \SS 
;\, |x- y_i|\le 1/2^{\p } \text{ for some } y_i \}$, where 
$ \mathsf{y}=\sum_i\delta_{y_i}$, 
\begin{align}
\label{:40F}&
 \bNrsp (x,\mathsf{y}) = 0
\quad \text{ for $ (x,\mathsf{y}) \in (\S \ts \SS )_{r,\p +1} $}
,\\ 
\label{:40g}&
 \bNrsp (x,\mathsf{y}) = \bbb _{r,s}^{N} (x,\mathsf{y}) 
 \quad \text{ for $ (x,\mathsf{y}) \not\in (\S \ts \SS )_{r,\p } $}
.\end{align}

The main requirements for $ \bN $ and $ \bNrsp $ are the following: 

\ms 

\noindent 
\As{I2} There exists a $ \phat $ such that 
$ 1 < \phat $ and that for each $ r \in \N $
\begin{align}\label{:40h}&
\limsupi{\nN } \int _{\SrSS }|\bN | ^{\phat } d\muNone < \infty 
.\end{align}
Furthermore, for each $ r , i \in \N $, there exists a constant $ \Ct{;40i}$ such that 
\begin{align}\label{:40i}& 
\sup_{\p \in\N } 
\sup_{\nN \in \N } E ^{\muNl } 
[\int_0^T | \bNrsp (X_t^{N,i},\XNidt ) | ^{\phat } dt ] \le 
\cref{;40i}
% \quad \text{ for all } 0 \le u, v \le T 
.\end{align}

\ms 

We set $ \SS _r^m = \{\mathsf{s}\, ; \mathsf{s}(\Sr ) = m\}$. 
Let $ \| \cdot \|_{\S \ts \SS _r^m} $ denote the uniform norm on 
$ \S \ts \SS _r^m $ and set 
$ L^{\phat }(\muNone _r ) = L^{\phat }(\Sr \ts \SS , \muNone )$. 
For a function $ f $ on $\S \ts \SS _r^m $ we denote by 
$ \nabla f = (\nablax  \check{f}, \nabla_{y_i} \check{f})$, 
where $ \check{f} $ is a function on $ \Sr \ts \Sr ^m $ such that 
$ \check{f}(x,(y_i)_{i=1}^m)$ is symmetric in $ (y_i)_{i=1}^m $
 for each $ x $ and $ f(x,\sum_i\delta_{y_i}) = \check{f}(x,(y_i)_{i=1}^m)$. 
We decompose $ \bNrs $ as 
\begin{align}\label{:40d}&
\bNrs = \bNrsp + \7 
\end{align}
and we assume: 

\medskip 
\noindent 
\As{I3} 
For each $ m ,\p , r , s \in \N $ such that $ r < s $, there exists 
$ \bbb _{r,s,\p } $ such that 
\begin{align}
\label{:41r} & 
\limi{\nN } \| \bbb _{r,s,\p }^{\nN } - \bbb _{r,s,\p } 
\|_{\S \ts \SS _r^m }
= 0 
.\\
\intertext{Moreover, $ \bbb _{r,s,\p }^{\nN } $ are differentiable in $ x $ and satisfying the bounds: }
\label{:41U}& 
\supN \| \nabla \bbb _{r,s,\p }^{\nN } \|_{\S \ts \SS _r^m } < \infty 
,\\\label{:41s}&
\limi{\p } \supN \| \bNrsp - \bNrs \|_{ L^{\phat }( \muNone _r ) } =0 
.\end{align}
Furthermore, we assume for each $ i , r < s \in \N $ 
\begin{align}\label{:41t}&
\limi{\p } \limsupi{\nN } 
E ^{\muNl } 
[\int_0^T | \{ \bNrsp - \bNrs \} (X_t^{N,i},\XNidt ) | ^{\phat }
dt ] = 0 
,\\\label{:41tt}&
\limi{\p }
E ^{\mul } 
[\int_0^T | \{ \brsp - \brs \} (X_t^{i},\Xidt ) | ^{\phat }
dt ] = 0 
,\end{align}
where $ \bbb _{r,s} $ is such that 
\begin{align}\label{:50l}&
 \bbb _{r,s} (x,\mathsf{y}) = 
\limi{\nN } \bbb _{r,s}^{\nN } (x,\mathsf{y}) 
\quad \text{ for each } (x,\mathsf{y}) \in \bigcup_{\p \in \N }
 (\S \ts \SS )_{r,\p } ^c 
.\end{align}
\begin{rem}\label{r:J3} 
We see that $ \bigcup_{\p \in \N }
 (\S \ts \SS )_{r,\p } ^c = \{\Sr ^c \ts \SS \} \cup 
\{ (x,\mathsf{y}) ; x \not= y_i \text{ for all }i\} $ by definition and 
$ \bbb _{r,s} (x,\mathsf{y}) = 0 $ for $ x \not\in \Sr $ by \eqref{:40b}.  
The limit in \eqref{:50l} exists because of 
 \eqref{:40F}, \eqref{:40g}, and \eqref{:41r}. 
\end{rem}

\noindent 
\As{I4} 
There exists a $ \btail \in C (\S ; \Rd )$ independent of $ r \in \N $ and 
$ \mathsf{s} \in \SSS $ such that 
\begin{align} \label{:41u} & 
\limi{s} \limsupi{\nN } \| \bNrstail - \btail 
\|_{ L^{\phat }( \muNone _r ) } 
 =0 
.\end{align}
Furthermore, for each $ r , i \in \N $: 
\begin{align}\label{:41v}&
\limi{s} \limsupi{\nN } 
 E ^{\muNl } 
[\int_0^T | ( \bNrstail - \btail ) (X_t^{N,i},\XNidt ) | ^{\phat }
dt ] = 0 
.\end{align}

We remark that $ \btail $ is automatically independent of $ r $ 
for consistency \eqref{:40g}. 
By assumption, $ \btail = \btail (x)$ is a function of $ x $. 
From \eqref{:40a} and \eqref{:40d} we have 
\begin{align}\label{:41w}&
\bN = \bNrsp + \btail + \{\7 \} + \{ \6 -\btail \} 
.\end{align}
In \As{I3} and \As{I4}, we have assumed that the last two terms 
$ \{\7 \} $ and $ \{ \6 -\btail \} $ in \eqref{:41w} are asymptotically negligible. 

Under these assumptions, we prove in \lref{l:50} that 
there exists $ \bbb $ such that for each $ r \in \N $
\begin{align}
\label{:41a}&
\limi{s} \| \brs - \bbb \|_{ L^{\phat }( \muNone _r ) } = 0 
.\end{align}
We assume: \\ 
\As{I5} For each $ i , r \in \N $ 
\begin{align}\label{:41z}&
\limi{s }
E ^{\mul } 
[\int_0^T | ( \brs - \mathsf{b} ) (X_t^{i},\Xidt ) | ^{\phat }
dt ] = 0 
.\end{align}

We say a sequence $ \{ \mathbf{X}^{\nN } \} $ of $ C([0,T];\S ^{\nN })$-valued 
random variables is tight if for any subsequence we can choose a subsequence 
denoted by the same symbol such that 
$ \{ \mathbf{X}^{\nN ,m} \}_{\nN \ge m } $ is convergent in law 
in $ C([0,T]; \S ^m )$ for each $ m \in \N $. 
With these preparations, we state the main theorem in this section. 
\begin{thm} \label{l:21}
Assume \As{H1}--\As{H4} and \0. 
Then, $ \{ \mathbf{X}^{\nN } \} _{\nN \in \N } $ 
is tight in $ C([0,T];\S ^{\N } )$ and, any limit point 
$ \mathbf{X} = (X^i)_{i\in\N } $ of $ \{ \mathbf{X}^{\nN } \} _{\nN \in \N } $ 
is a solution of the ISDE 
\begin{align}\label{:41b}&
dX_t^i = \sigma (X_t^i,\mathsf{X}_t^{\diai }) dB_t^i + 
\{\bbb (X_t^i,\mathsf{X}_t^{\diai })+ 
\btail (X_t^i)\} dt 
.\end{align}
\end{thm}

\begin{rem}\label{r:49}
If diffusion processes are symmetric, we can dispense with 
\eqref{:40z}, \eqref{:40i}, \eqref{:41t}, \eqref{:41tt}, 
\eqref{:41v}, and \eqref{:41z} as we see in Subsection \ref{s:23}. 
Indeed, using the Lyons-Zheng decomposition 
we can derive these from {\em static} conditions 
\As{H4}, \eqref{:40h}, \eqref{:41r}, \eqref{:41s}, \eqref{:41u}, and \eqref{:41a}.  
We remark that we can apply \tref{l:21} to non-symmetric diffusion processes by assuming these {\em dynamical} conditions. 
\end{rem}

\subsection{Finite-particle approximations (reversible case)}\label{s:23}

For a subset $ A $, we set $ \map{\pi_{A}}{\SS }{\SS } $ by 
 $ \pi_{A} (\mathsf{s}) = \mathsf{s} (\cdot \cap A )$. 
 We say a function $ f $ on $ \SS $ is local if 
$ f $ is $ \sigma [\pi_{K}]$-measurable for some compact set $ K $ in $ \S $. 
 For a local function $ f $ on $ \SS $, we say $ f $ is smooth if $ \check{f}$ is smooth, where 
 $ \check{f}(x_1,\ldots )$ is a symmetric function such that 
 $ \check{f}(x_1,\ldots ) = f (\mathsf{x})$ for $ \mathsf{x} = \sum _i \delta _{x_i}$. 
 Let $ \di $ be the set of all bounded, local smooth functions on $ \SS $. 
We write $ f \in L_{\mathrm{loc}}^p(\muone )$ if 
$ f \in L^p(\Sr \ts \SS , \muone )$ for all $ r \in\N $. 
Let 
$ C_{0}^{\infty}(\S )\ot \di = \{ \sum_{i=1}^Nf_i (x) g_i (\mathsf{y})\, ;\, f_i \in C_{0}^{\infty}(\S ),\, g_i \in \di ,\, N \in \N \} $  
denote the algebraic tensor product of $ C_{0}^{\infty}(\S ) $ and $\di $. 

\begin{dfn}\label{d:21}
A $ \Rd $-valued function $ \dmu \in L_{\mathrm{loc}}^1(\muone )$ is called 
 {\em the logarithmic derivative} of $\mu $ if, for all 
$\varphi \in C_{0}^{\infty}(\S )\ot \di $, 
\begin{align}&\notag % \label{:21i}&
 \int _{\S \times \SS } 
 \dmu (x,\mathsf{y})\varphi (x,\mathsf{y}) 
 \muone (dx d\mathsf{y}) = 
 - \int _{\S \times \SS } 
 \nablax  \varphi (x,\mathsf{y}) \muone (dx d\mathsf{y}) 
.\end{align}
\end{dfn}

\begin{rem}\label{r:21}
\thetag{1} 
The logarithmic derivative $ \dmu $ is determined uniquely (if exists). 
\\\thetag{2}
If the boundary $ \partial \S $ is nonempty and particles 
hit the boundary, then $ \dmu $ would contain a term arising from the boundary condition. 
 For example, if the Neumann boundary condition is imposed on the boundary, 
 then there would be local time-type drifts. 
We shall later assume that particles never hit the boundary, 
and the above formulation is thus sufficient in the present situation. 
\\
\thetag{3} 
A sufficient condition for the explicit expression of 
the logarithmic derivative of point processes is given in \cite[Theorem 45]{o.isde}. 
Using this, one can obtain the logarithmic derivative of point processes appearing in random matrix theory such as sine$_{\beta } $, Airy$_{\beta }, $ ($ \beta=1,2,4$), Bessel$_{2,\alpha } $ ($ 1\le \alpha $), 
and the Ginibre point process (see Examples in \sref{s:5}). For canonical Gibbs measures with Ruelle-class interaction potentials, one can easily calculate the logarithmic derivative employing DLR equation \cite[Lemma 10.10]{o-t.tail}. 
\end{rem}

We assume:

\bs 

\noindent 
\As{J1} 
Each $ \muN $ has a logarithmic derivative $ \dmuN $, 
and the coefficient $ \bbb ^{\nN } $ is given as 
\begin{align}\label{:21j}&
\bbb ^{\nN }=\frac{1}{2}\{\nablax \aN + \aN \dmuN \}
.\end{align}
Furthermore, the vector-valued functions $ \{\nablax \aN \}_{\nN }$ are continuous and converge to $ \nablax \aaa $ uniformly on each $ \Sr \ts \SS $, where 
$ \nablax \aN $ is the $ d $-dimensional column vector such that 
\begin{align}\label{:21k}&
\nablax \aN (x,\mathsf{y}) = 
{}^{t}\Big( 
\sum_{i=1}^d \PD{}{x_i}\aN _{1i} (x,\mathsf{y}),\ldots,
\sum_{i=1}^d \PD{}{x_i}\aN _{di} (x,\mathsf{y}) 
\Big)
.\end{align}

\begin{rem}\label{r:26} 
From \As{J1} we see that the delabeled dynamics 
$ \mathsf{X}^{\nN } = \sum_{i=1}^{\nN } \delta_{X^i}$ 
of $ \mathbf{X}^{\nN } $ is reversible with respect to $ \muN $. 
Thus \As{J1} relates the measure $ \muN $ 
with the labeled dynamics $ \mathbf{X}^{\nN } $. 
For each $ \nN < \infty $, $\XX ^{\nN }$ has a reversible measure. 
Indeed, the symmetrization $ (\muN \circ \labN ^{-1})_{\mathrm{sym}}$ 
of $\muN \circ \labN ^{-1} $ is a reversible measure of $\XX ^{\nN }$ 
as we see for $ \muNcheck $ in Introduction, where 
$ (\muN \circ \labN ^{-1})_{\mathrm{sym}}= \frac{1}{\nN ! }
\sum_{\sigma \in \mathfrak{S}_N}(\muN \circ \labN ^{-1})\circ \sigma ^{-1} $ and 
$ \mathfrak{S}_N$ is the symmetric group of order $ N $. 
When $ \nN = \infty $, $ \mathbf{X}$ does not have any reversible measure 
in general. For example, infinite-dimensional Brownian motion $ \mathbf{B}=(B^i)_{i\in\N }$ on $ (\Rd )^{\N }$ has no reversible measures. 
We also remark that the Airy$_{\beta }$ ($ \beta = 1,2,4$) interacting Brownian motion 
defined by \eqref{:10H} has a reversible measure given by 
$  \mu _{\mathrm{Airy}, \beta }\circ\lab ^{-1}$ 
 with label $ \lab (\mathsf{s}) = (s_1,s_2,\ldots)$ such that $ s_i > s_{i+1}$ 
for all $ i \in \N $ because $ \lab $ gives a bijection from (a subset of) $ \SS $ to $ \R ^{\N }$ defined for $  \mu _{\mathrm{Airy}, \beta } $-a.s.\  $ \mathsf{s}$, and thus the relation 
$ \mathbf{X}_t = \lab (\mathsf{X}_t)$ holds for all $ t $. 
\end{rem}

We prove that convergence of the logarithmic derivative 
implies weak convergence of the solutions of the associated SDEs. 
Each logarithmic derivative $ \dmuN $ belongs to a different $ L^p $-space 
$ L^p (\muNone )$, and $ \muNone $ are mutually singular. 
Hence we decompose $ \dmuN $ to define a kind of $ L^p $-convergence.

Let $ \map{\uu ,\ \uN ,\ \ww }{\S }{\Rd }$ and 
$\map{\g ,\, \gN ,\, \vvv ,\, \vN }{\S ^{2} }{\Rd }$ 
 be measurable functions. We set 
\begin{align}
\label{:21l}&
\ggs (x,\mathsf{y})= \vs + \sum_{i} \chi _s (x-y_i) \g (x,y_i)
,\\&\notag 
\ggNs (x,\mathsf{y})= \vNs + \sum_{i} \chi _s (x-y_i) \gN (x,y_i)
,\\&\notag 
\rrNs (x,\mathsf{y})= \vNsinfty + \sum_{i}(1- \chi _s (x-y_i))\gN (x,y_i) 
,\end{align}
where $ \mathsf{y}=\sum_{i}\delta_{y_i}$ and 
$ \chi _s \in C_0^{\infty} (\S )$ is a cut-off function such that 
$ 0 \le \chi _s \le 1 $, $ \chi _s (x) = 0 $ for $ |x| \ge s + 1 $, and 
$ \chi _s (x) = 1$ for $ |x| \le s $. 
We assume the following. 

\medskip

\noindent 
\As{J2} Each $ \muN $ has a logarithmic derivative $ \dmuN $ such that 
\begin{align}\label{:21m}&
\dmuN (x,\mathsf{y})= \uN (x) + \ggNs (x,\mathsf{y}) + \rrNs (x,\mathsf{y})
.\end{align}
Furthermore, we assume that 
\begin{itemize}

\item[(1)]
$\uN $ are in $ C^1 (\S ) $. Furthermore, 
$ \uN $ and $ \nabla \uN $ converge uniformly to $ u $ and $ \nabla u $, respectively, on each compact set in $ \S $.

\item[(2)] 
For each $ s \in \N $, $\vNs $ are in $ C^1 (\S ) $. 
Furthermore, functions 
$\vNs $ and $ \nablax \vNs $ converge uniformly to $\vs $ and 
$ \nablax \vs $, respectively, on each compact set in $ \S $. 

\item[(3)]
$ g^{\nN }$ are in $ C^1 (\S ^2 \cap \{x \not= y \}) $. 
Furthermore, 
$ g^{\nN }$ and $\nablax  g^{\nN }$ converge uniformly to 
$ g $ and $\nablax  g $, respectively, on $ \S ^2 \cap \{ |x-y| \ge 2^{-\p } \} $ for each $ \p >0 $. 
In addition, for each $ r \in \N $, 
\begin{align}\label{:21q}&
\limi{\p } \limsupi{\nN } 
\int_{x \in \Sr , |x-y| \le 2^{-\p } } 
\chi _s (x-y) | \gN (x,y) | ^{\phat }\, \rho _x^{\nN ,1}(y) dxdy = 0 
,\end{align}
where $ \rho _x^{\nN ,1} $ is a one-correlation function of the reduced Palm measure $ \mu _x^{\nN }$. 

\item[(4)] There exists a continuous function $ \map{w }{\S }{\R }$ such that 
\begin{align}\label{:21r}& 
\limi{s}\limsupi{N} 
 \int_{\Sr \times \SSS } |\rrNs (x,\mathsf{y}) - \www |^{\phat } d\muNone 
= 0 , \quad \ww \in L^{\phat }_{\mathrm{loc}} (\S ,dx)
.\end{align}
\end{itemize}

Let $ p $ be such that $ 1 < p < \phat $. 
Assume \As{H1} and \As{J2}. 
Then from \cite[Theorem 45]{o.isde} we see that 
the logarithmic derivative $\dmu $ of $ \mu $ exists in 
$ \Lploc (\muone )$ and is given by 
\begin{align} \label{:21v}
\dmu (x,\mathsf{y})= 
\uu (x) + \mathsf{g} (x,\mathsf{y}) + \www 
.\end{align}
Here $ \mathsf{g} (x,\mathsf{y})= \limi{s} \ggs (x,\mathsf{y}) $ and 
the convergence of $ \lim \ggs $ takes place in $ \Lploc (\muone )$. 
We now introduce the ISDE of $ \mathbf{X}=(X^i)_{i\in\N }$: 
\begin{align}\label{:21w}&
dX_t^i = \sigma (X_t^i,\mathsf{X}_t^{\diai }) dB_t^i 
%+ \sigma ^{\tail }(X_t^i) \}dB_t^i 
 + 
\frac{1}{2} \{
\nablax  \aaa (X_t^i,\mathsf{X}_t^{\diai }) 
+ 
\aaa (X_t^i,\mathsf{X}_t^{\diai })
\dmu (X_t^i,\mathsf{X}_t^{\diai }) 
\} dt 
\\\label{:21W}&
\mathbf{X}_0 = \mathbf{s}
.\end{align}
Here $ \nablax  \aaa $ is defined similarly as \eqref{:21k}. 
If $ \sigma $ is the unit matrix and \As{J2} is satisfied, we have 
\begin{align}\label{:21x}
dX_t^i &= dB_t^i 
+ \frac{1}{2} \{ 
u (X_t^i) + w (X_t^i) 
 + 
\mathsf{g} (X_t^i,\mathsf{X}_t^{\diai }) 
\} dt 
.\end{align}
In the sequel, we give a sufficient condition 
for solving ISDE \eqref{:21w} (and \eqref{:21x}). 

Let $ \mathbb{D}$ be the standard square field on $ \SS $ such that 
for any $ f , g \in \di $ and $ \mathsf{s}=\sum_i\delta_{s_i}$
\begin{align*}&
\mathbb{D}[f,g] (\mathsf{s})= 
\half \{\sum_i \nabla _i{\check{f}} \cdot \nabla _i{\check{g}}\} \, 
 (\mathsf{s})
,\end{align*}
where $ \cdot $ is the inner product in $ \Rd $. Since the function 
$ \sum_i 
\nabla _i{\check{f}} (\mathbf{s})\cdot \nabla _i{\check{g}} (\mathbf{s})$, 
where $ \mathbf{s}=(s_i)_i$ and $ \mathsf{s}= \sum_i \delta_{s_i}$, 
is symmetric in $ (s_i)_i$, we regard it as a function of $ \mathsf{s}$. 
We set $ \Lm = L^2(\SS ,\mu )$ and let 
\begin{align*}&
\mathcal{E}^{\mu }(f,g) = \int_{\SS } \mathbb{D}[f,g] (\mathsf{s}) \mu (d\mathsf{s}), 
\\ & 
 \di ^{\mu } =\{ f \in \di \cap \Lm \, ;\, 
\mathcal{E}^{\mu }(f,f) < \infty \} 
.\end{align*}
We assume: 

\ms 

\noindent 
\As{J3} \quad 
$ (\mathcal{E}^{\mu }, \mathcal{D}_{\circ }^{\mu } )$ is closable on $ \Lm $. 

\bs 

From \As{J3} and the local boundedness of correlation functions given by 
\As{H1}, we deduce that 
the closure $ (\mathcal{E}^{\mu }, \mathcal{D}^{\mu } )$ 
of $ (\mathcal{E}^{\mu }, \mathcal{D}_{\circ }^{\mu } )$ 
becomes a quasi-regular Dirichlet form \cite[Theorem 1]{o.dfa}. 
Hence, using a general theory of quasi-regular Dirichlet forms, 
we deduce the existence of 
the associated $ \SS $-valued diffusion $ (\mathsf{P},\mathsf{X})$ \cite{mr}. 
By construction, 
$ (\mathsf{P},\mathsf{X})$ is $ \mu $-reversible. 

If one takes $ \mu $ as Poisson point process with Lebesgue intensity, then the diffusion $ (\mathsf{P},\mathsf{X})$ thus obtained 
is the standard $ \SS $-valued Brownian motion $ \mathsf{B}$ such that 
$ \mathsf{B}_t=\sum_{i\in\N } \delta_{B_t^i}$, where 
$ \{ B^i \}_{i\in\N } $ are independent copies of the standard Brownian motions on $ \Rd $. This is the reason why we call $ \mathbb{D}$ the standard square field.

\bigskip

Let $ \mathrm{Cap}^{\mu }$ denote the capacity given by the Dirichlet space 
$ (\mathcal{E}^{\mu }, \mathcal{D}^{\mu }, \Lm )$ \cite{FOT.2}. 
Let 
\begin{align}& \notag %& \label{:21s}
\SSSsi =
\{ \sss \in \SSS \, ;\, \, \sss (x)\le 1 \text{ for all }x \in \S ,\, \, 
\sss (\S )= \infty \} 
\end{align}
and assume:

\medskip
\noindent
\As{J4}
$\mathrm{Cap}^{\mu } (\{\SSSsi \}^c) = 0 $. 

\ms 

\noindent 
Let $\erf (t)=\frac{1}{\sqrt{2\pi }} \int_t^\infty e^{ - {|x|^2}/{2}}\,dx$ be the error function. 
Let $ \Sr = \{ |x| < r \} $ as before. We assume: 

\medskip

\noindent
\As{J5}
There exists a $ Q >0 $ such that for each $ R >0 $ 
\begin{align}\label{:21t}
\liminfi{r} \sup_{\nN \in\N } 
\Bigl\{ \int_{\S _{r+R}} \rho ^{\nN ,1} (x)\,dx 
\Bigr\} \erf \Bigl(\frac{r}{\sqrt{(r+R ) Q }}\Bigr)= 0
.\end{align}
We write $ s_i=\labN (\mathsf{s})_i$ and assume for each $ r \in \N $
\begin{align}\label{:21T}&
 \limi{L} \limsupi{\nN } \sum_{i>L} 
\int_{\SS } \mathrm{Erf} (\frac{|s_i|-r}{\sqrt{\cref{;3}}T}) 
\muN (d\mathsf{s}) = 0 
.\end{align}

We remark that \eqref{:21T} is easy to check. Indeed, 
we prove in \lref{l:67} that, if $ s_i=\labN (\mathsf{s})_i$ is taken 
such that 
\begin{align}\label{:25g}&
|s_1|\le |s_2|\le \cdots 
,\end{align}
then \eqref{:21T} follows from \As{H1} and \eqref{:25f} below. 
\begin{align}\label{:25f}& \quad 
\limi{q}\limsupi{N} \int_{\S \backslash \S _{q}} 
\mathrm{Erf} (\frac{|x|-r}{\sqrt{\cref{;3}}T}) \rho^{\nN , 1} (x) dx =0 
% \quad \text{for each $ r \in \N $}
.\end{align}

Let $ \lab $ be the label as before. 
Let $ \mathbf{X} = (X^i)_{i\in\N }$ be a family of solution of 
\eqref{:21w} satisfying %\eqref{:21W} for 
$ \mathbf{X}_0 = \mathbf{s} $ for $ \mu \circ \lab ^{-1}$-a.s.\! $ \mathbf{s}$. 
We call $ \mathbf{X}$ satisfies $ \mu $-absolute continuity condition if 
\begin{align}\label{:21y}&
\mu _t \prec \mu \quad \text{ for all } t \ge 0 
,\end{align}
where $ \mu _t $ is the distribution of $ \mathsf{X}_t $ and 
$ \mu _t \prec \mu $ means $ \mu _t $ 
is absolutely continuous with respect to $ \mu $. 
Here $ \mathsf{X}_t = \sum_{i\in\N } \delta_{X_t^i}$, 
for $ \mathbf{X}_t = (X_t^i)_{i\in\N }$. By definition 
$ \mathsf{X} =\{ \mathsf{X}_t \} $ is the delabeled dynamics of 
$ \mathbf{X}$ and by construction $ \mathsf{X}_0 = \mu $ in distribution. 

We say ISDE \eqref{:21w} has $ \mu $-uniqueness of solutions in law 
if $ \mathbf{X}$ and $ \mathbf{X}'$ are solutions with the same initial distributions 
 satisfying the $ \mu $-absolute continuity condition, then 
 they are equivalent in law. 
We assume: 

\medskip 

\noindent 
\As{J6} 
 ISDE \eqref{:21w} has $ \mu $-uniqueness of solutions in law.

\medskip 

Let $ \mathbf{X}^{N} $ be a solution of \eqref{:20m}. From \eqref{:21j} 
we can rewrite \eqref{:20m} as 
\begin{align}\label{:21z}&
dX_t^{N,i} = \sigmaN (X_t^{N,i} ,\mathsf{X}_t^{\nN , \diai }) dB_t^i 
%+ \sigma ^{\tail }(X_t^i) \}dB_t^i 
 + 
\frac{1}{2} \{
\nablax  \aN 
+ 
\aN 
\dmuN 
\} (X_t^{N,i} ,\mathsf{X}_t^{\nN , \diai }) dt 
.\end{align} 
We set $ \mathbf{X}^{\nN , m }=  (X^{\nN ,1},X^{\nN ,2},\ldots,X^{\nN ,m}) $ 
$ 1 \le m \le \nN $ and 
$ \mathbf{X}^{m}= (X^{1},X^{2},\ldots,X^{m})$. 
We say $ \{ \mathbf{X}^{\nN } \} $ is tight in $ C([0,\infty );\S ^{\N })$ 
if each subsequence $ \{ \mathbf{X}^{\nN '} \} $ contains a subsequence 
$ \{ \mathbf{X}^{\nN ''} \} $ such that 
$ \{ \mathbf{X}^{\nN '',m} \} $ is convergent weakly in $ C([0,\infty );\S ^m )$ 
for each $ m \in \N $. 

\begin{thm}\label{l:22}
Assume \As{H1}--\As{H4} and \As{J1}--\As{J5}. 
Assume that $  \mathbf{X}_0^{\nN }= \muN \circ \labN ^{-1} $ in distribution. 
Then $ \{ \mathbf{X}^{\nN } \} $ is tight in $ C([0,\infty );\S ^{\N })$ and 
each limit point $ \mathbf{X}$ of $ \{ \mathbf{X}^{N} \} $ is a solution of 
\eqref{:21w} with initial distribution $ \mul $. 
Furthermore, if we assume \As{J6} in addition, then for any $m \in \mathbb{N}$ 
\begin{align}\label{:21a}
\limi{\nN } \mathbf{X}^{\nN , m } = \mathbf{X}^m 
 \quad\text{ weakly in } C([0,\infty) ,\S ^m)
.\end{align}
\end{thm}

 \begin{rem}\label{r:21a} 
 To prove \eqref{:21a} it is sufficient to prove the convergence in $ C([0,T];\S ^m )$ for each $ T \in \N $. We do this in the following sections. 
 \end{rem}
% 

% 
% Combining \tref{l:22} with \cite{o-t.tail}, we obtain: 
% \begin{rem}\label{r:21z} 
% Assumption \As{J6} is not overdemanding. 
% In addition to \As{H1}--\As{H4} and \As{J1}--\As{J5}, 
% we assume \As{E1}, \As{F1}, \As{F2} in  \cite{o-t.tail}, 
% and tail triviality of $ \mu $.  Then \As{J6} is satisfied. 
% Hence from \tref{l:22} we deduce that 
% each limit point $ \mathbf{X}$ of $ \{ \mathbf{X}^{N} \} $ is a solution of 
% \eqref{:21w} with initial distribution $ \mul $ and \eqref{:21a} holds. 
%% 
% \end{rem}

\begin{rem}\label{r:k4} \thetag{1} 
A sufficient condition for \As{J3} is obtained in \cite{o.rm, o.rm2}. 
Indeed, if $ \mu $ is a $ (\Phi ,\Psi )$-quasi-Gibbs measure 
with upper semi-continuous potential $ (\Phi ,\Psi )$, 
then \As{J3} is satisfied. This condition is mild and is satisfied by 
all examples in the present paper. 
We refer to \cite{o.rm, o.rm2} for the definition of quasi-Gibbs property. 
\\
\thetag{2}
From the general theory of Dirichlet forms, 
we see that \As{J4} is equivalent to the non-collision of particles \cite{FOT.2}. 
We refer to \cite{inu} for a necessary and sufficient condition of this non-collision property of interacting Brownian motions in finite-dimensions, which gives a sufficient condition of non-collision in infinite dimensions. 
We also refer to \cite{o.col} for a sufficient condition for non-collision property of interacting Brownian motions in infinite-dimensions applicable to, in particular, determinantal point processes. 
\\\thetag{3} 
From \eqref{:21t} of \As{J5}, we deduce that each tagged particle $ X^i $ 
does not explode \cite{FOT.2,o.tp}. We remark that the delabeled dynamics 
$ \mathsf{X}=\sum_i \delta_{X^i}$ are $ \mu $-reversible, and they thus 
never explode. Indeed, as for configuration-valued diffusions, explosion occurs if and only if infinitely many particles gather in a compact domain, 
so the explosion of tagged particle does not imply that of the configuration-valued process. 
\\\thetag{4} 
It is known that, if we suppose \As{H1}, \As{J1}--\As{J5}, then ISDE \eqref{:21w} has a solution for $ \mul $-a.s. $ \mathbf{s}$ satisfying 
the non-collision and non-explosion property \cite{o.isde}. 
Indeed, let $ \mathbf{X}=(X^i)$ be the $ \SN $-valued continuous process 
consisting of 
tagged particles $ X^i$ of the delabeled diffusion process $ \mathsf{X}=\sum_{i\in\N } \delta_{X^i}$ given by the Dirichlet form of \As{J3}. 
Then from \As{J4} and \As{J5} \eqref{:21t} we see 
$ \mathbf{X}$ is uniquely determined by its initial starting point. 
It was proved that $ \mathbf{X}$ is a solution of 
\eqref{:21w} in \cite{o.isde}. 
\end{rem}

\begin{rem}\label{r:K5} 
Assumption \As{J6} follows from tail triviality of $ \mu $ \cite{o-t.tail}, where tail triviality  of $ \mu $ means that the tail $ \sigma $-field 
$ \mathcal{T}=\bigcap_{r=1}^{\infty} \sigma [\pi _{\Sr ^c}]$ is $ \mu $-trivial. 
Indeed, from tail triviality of $ \mu $ and marginal assumptions (\As{E1}, \As{F1}, and \As{F2} in \cite{o-t.tail}), we obtain \As{J6}. 
Tail triviality holds for all determinantal point processes \cite{o-o.tt} and 
grand canonical Gibbs measures with sufficiently small inverse temperature 
$ \beta > 0 $. 
\end{rem}

\section{Proof of \tref{l:21}}\label{s:3} 

The purpose of this section is to prove \tref{l:21}. 
We assume the same assumptions as \tref{l:21} 
throughout this section. We begin by proving \eqref{:41a}. 

\begin{lem} \label{l:50} 
\eqref{:41a} holds. 
\end{lem}
\begin{proof}
From \As{H1} and \eqref{:41r}, we obtain 
\begin{align} \label{:50z}&\quad \quad \quad 
\limi{N} \bNrsp = \, \brsp 
\quad \text{ for $ \muonebar $-a.s.\! and in $ \Lp $. }
\end{align}
We next prove the convergence of $ \{ \brsp \} $ as $ \p \to \infty $. 
Note that 
\begin{align}\label{:50i}&
\| \brsp - \brsq \|_{ \Lp }
 \\\le \, & \notag
\| \brsp - \bNrsp \|_{ \Lp } + 
\| \bNrsp - \bNrsq \|_{ \Lp } + 
\| \bNrsq - \brsq \|_{ \Lp }
.\end{align}
From \eqref{:41s} for each $ \epsilon $ there exists a $ \p _0$ such that 
for all $ \p ,\q \ge \p _0 $ 
\begin{align}\label{:50p}&
\supN \| \bNrsp -\bNrsq \|_{ L^{\phat }( \muNone _r ) } < \epsilon 
\end{align}
By \eqref{:50z} there exists an $ \nN = \nN_{\p ,\q }$ such that 
\begin{align}\label{:50q}&
 \| \brsp - \bNrsp \|_{ \Lp } < \epsilon ,\quad 
 \| \brsq - \bNrsq \|_{ \Lp } < \epsilon 
.\end{align}
Putting \eqref{:50p} and \eqref{:50q} into \eqref{:50i}, we deduce that 
$ \{ \brsp \}_{\p \in \N } $ is a Cauchy sequence in $ \Lp $. 
Hence from \eqref{:40g}, \eqref{:41s}, and \eqref{:50l} we see 
\begin{align}\label{:50a}
 \limi{\p } \brsp = \, &\brs \quad \text{ in } \Lp 
.\end{align}

Recall that 
$ \bbb _{r,s}^{\nN } = 
 \mathrm{E}^{\muNonebar } [ \bbb ^{N}|\mathcal{F}_{\rrr ,s} ] $ 
by \eqref{:40a}. Then, because 
$ \mathcal{F}_{\rrr ,s} \subset \mathcal{F}_{\rrr ,s+1}$, we have 
\begin{align}\label{:50J}&
\bbb _{r,s}^{\nN } = 
 \mathrm{E}^{\muNonebar } [ \bbb _{r,s+1}^{N}|\mathcal{F}_{\rrr ,s} ] 
.\end{align}
From $ \bbb _{r,s}^{\nN } = 
 \mathrm{E}^{\muNonebar } [ \bbb ^{N}|\mathcal{F}_{\rrr ,s} ] $ 
we have 
\begin{align}& \notag 
 \| \bbb _{r,s}^{\nN } \|_{\LNp } \le \| \bN \|_{\LNp }
\end{align}
From this and \eqref{:40h} we obtain 
\begin{align}\label{:50X}&
\sup_{r < s } \limsupi{\nN } \| \bbb _{r,s}^{\nN } \|_{\LNp } \le 
\limsupi{\nN } \| \bN \|_{\LNp } < 
\infty 
.\end{align}
Combining \eqref{:50l}, \eqref{:50J} and \eqref{:50X}, we have 
\begin{align}\label{:50n}&
\bbb _{r,s} = 
\limi{\nN}  \bbb _{r,s}^{\nN } =  \limi{\nN} 
\mathrm{E}^{\muNonebar } [\bbb _{r,s+1}^{N}|\mathcal{F}_{\rrr ,s} ] =
\mathrm{E}^{\muonebar } [ \bbb _{r,s+1}|\mathcal{F}_{\rrr ,s} ]
.\end{align}

From \As{H1}, \eqref{:50l}, \eqref{:50X}, and Fatou's lemma, we see that 
\begin{align}\label{:50m}& 
 \sup_{r < s } \| \bbb _{r,s} \|_{\Lp } \le 
 \sup_{r < s } \liminfi{\nN } \| \bbb _{r,s}^{\nN } \|_{\LNp }
 < 
\infty 
.\end{align}
From \eqref{:50n} we deduce that 
$ \{ \bbb _{r,s} \}_{s=r+1}^{\infty} $ is martingale in $ s $. 
Applying the martingale convergence theorem to $ \{\brs \}_{s=r+1}^{\infty}$ and 
using  \eqref{:50m}, we deduce that there exists a $ \br $ such that 
\begin{align}\label{:50o}&
 \brs = \mathrm{E}^{\muonebar } [ \br |\mathcal{F}_{\rrr ,s}] 
\end{align}
and that 
\begin{align}&\notag %\label{:50o}&
\limi{s} \bbb _{r,s} = \br 
\quad \text{for $ \muonebar $-a.s.\! and in }
L^{\phat }(\muonebar )
.\end{align}
By the consistency of $ \{ \muonebar \}_{r \in \N } $ in $ r $, 
the function $ \br $ in \eqref{:50o} can be taken to be independent of $ r $. 
This together with \eqref{:50a} completes the proof of \eqref{:41a}. 
\end{proof}

\bs

We proceed with the proof of the latter half of \tref{l:21}. 
Recall SDE \eqref{:20m}. Then 
\begin{align}\label{:52a} & 
X_t^{N,i} - X_0^{N,i} = %\sN (X_t^{N,i},\XNidt )dB_t^i 
\int_0^{t} \sN (X_u^{N,i},\XNidu )dB _u^i + 
\int_0^t \bbb ^{N} (X_u^{N,i},\XNidu ) du 
.\end{align}
Using the decomposition in \eqref{:41w}, we see from \eqref{:52a} that 
\begin{align}\label{:52b} 
X_t^{N,i} - X_0^{N,i} =\, &
\int_0^{t} \sN (X_u^{N,i},\XNidu )dB _u^i
 + 
\int_0^t \{\bNrsp + \btail \} (X_u^{N,i},\XNidu ) du 
\\ \notag + &
\int_0^t \Big[ \{\7 \} 
+
\{ \6 - \btail \} \Big] (X_u^{N,i},\XNidu ) du 
.\end{align}

Let $ \partial _{i,j} = \PD{}{x_{i,j}}$, 
$ x_i = (x_{i,j})_{j=1}^d \in \Rd $, 
and 
$ \mathbf{x}_m = (x_i)_{i=1}^m \in (\Rd )^m $. 
Set $ \nabla _i =(\partial _{i,j} )_{j=1}^d$. 
Let $ \hhh \in C_0^{\infty} (\S ^m) $ and 
$ \aN _i\nabla _i\nabla _i \hhh (\mathbf{x}_m) = 
\sum_{k,l=1}^d \aN _{kl} (x_i) \partial_{i,k}\partial_{i,l} 
\hhh (\mathbf{x}_m)$. 
Applying the It$ \hat{\mathrm{o}}$ formula to $ \hhh $ and \eqref{:52b}, and 
putting $ \mathbf{X}_t^{\nN , m } = (X_t^{N,1},\ldots,X_t^{N,m})$, we deduce that 
\begin{align}\label{:52d}&
\hhh (\mathbf{X}_t^{N,m} ) - \hhh (\mathbf{X}_0^{N,m} ) = 
\sum_{i=1}^m \Big ( \int_0^{t}
\nc \cdot \sN (X_u^{N,i},\XNidu )dB _u^i 
 \\ \notag & 
+ 
{\int_0^t \half 
\aN _i\nabla _i\nabla _i \hhh (\mathbf{X}_u^{\nN , m } ) + 
 \{\bNrsp + \btail \} (X_u^{N,i},\XNidu ) \cdot \nc du }
 \Big) 
\\ \notag & + 
\sum_{i=1}^m \int_0^t \nc \cdot \{ \7 \} (X_u^{N,i},\XNidu ) du 
\\ \notag & + 
\sum_{i=1}^m \int_0^t 
\nc 
 \cdot 
 \{\6 - \btail \} (X_u^{N,i},\XNidu ) du 
.\end{align}
We set 
\begin{align}&\notag %\label{:52c}&
\QN = 
\sum_{i=1}^m \int_0^T \Big| \{ \7 \} (X_u^{N,i},\XNidu ) 
\Big| 
du 
,\\ &\notag %\label{:52e}&
\RN = 
\sum_{i=1}^m \int_0^T \Big|
 \{\6 - \btail \} (X_u^{N,i},\XNidu ) \Big|du 
.\end{align}

\begin{lem} \label{l:56}For each $ m , r < s \in \N $ 
\begin{align}& \notag %\label{:56a}&
\limi{\p } \limsupi{\nN } \mathrm{E} ^{\muNl }
\big[ (\QN )^{\phat }\big] 
= 0 
,\\& \notag %\label{:56b}&
\limi{s} \limsupi{\nN } \mathrm{E} ^{\muNl }
\big[ (\RN )^{\phat }\big] 
= 0 
.\end{align}
\end{lem}
\begin{proof}
\lref{l:56} follows from \eqref{:41t} and \eqref{:41v} immediately. 
\end{proof}

Let $ \Xi = \S ^m \ts (\R ^{d^2})^m \ts (\Rd )^m $ and $ \hhh \in C_0^{\infty} (\S ^m) $. 
Let $ \map{F }{C([0,T];\Xi )}{C([0,T]; \R )}$ such that 
\begin{align}\label{:52p}
F (\xi , \eta , \zeta )& (t) = \hhh (\xi (t)) - \hhh (\xi (0) ) - 
\int_0^t \sum_{i=1}^m \zeta _i (u) \cdot \nabla _i \hhh (\xi (u)) du 
\\ \notag &
- 
\int_0^t \sum_{i=1}^m \Big(
 \half \eta _{i} (u) \Delta _i \hhh (\xi (u)) + 
 \btail (\xi _i(u)) \cdot \nabla _i \hhh (\xi (u)) \Big) du 
,\end{align}
where $ \xi = (\xi _i)_{i=1}^m $, 
$ \eta = (\eta _i)_{i=1}^m $, $ \eta _i = (\eta _{i,kl})_{k,l=1}^d$, 
$ \zeta = (\zeta _i)_{i=1}^m $, and 
$ \Delta _i = \sum_{j=1}^d \partial_{i,j} ^2 $. 

As $ \hhh \in C_0^{\infty} (\S ^m) $ and $ \btail \in C (\S ^m) $ by definition, we see that $ F $ satisfies the following. 
\\
\thetag{1} $ F $ is continuous. \\
\thetag{2} $ F (\xi , \eta , \zeta ) $ is bounded in 
$ (\xi , \eta ) $ for each $ \zeta $, and linear in $ \zeta $ for each $ (\xi , \eta )$.

\medskip

Let 
$ \mathbf{A}^{\nN ,m}= ( \mathsf{A}^{\nN ,i})_{i=1}^m $ and 
$ \BNrspm = ( \BNirsp )_{i=1}^m $ 
such that 
\begin{align}\label{:52g}& 
 \ANirs (t) = \aNrs (X_t^{N,i},\XNidt ) 
,\quad %\\\notag & 
\BNirsp (t) = \bNrsp (X_t^{N,i},\XNidt ) 
.\end{align}
Then we see from \eqref{:52d}--\eqref{:52g} that for each $ m \in \N $
\begin{align}\label{:52h}&
\Big|F (\mathbf{X}^{N,m}, \mathbf{A}^{\nN ,m} , \BNrspm ) -
\sum_{i=1}^m \int_0^{\cdot }
\nc \cdot \sN (X_u^{N,i},\XNidu )dB _u^i \Big|\\ \notag &
 \le 
 \cref{;52h}\{ \QN + \RN \} 
,\end{align}
where $ \Ct{;52h} =\cref{;52h}(\psi )$ is the constant such that 
$ \cref{;52h}= \max_{i=1}^m \|\nabla _i \psi \|_{\S ^m}$ 
($ \| \cdot \|_{A}$ is the uniform norm over $ A$ as before). 
We take the limit of each term in \eqref{:52h} in the sequel. 

\begin{lem} \label{l:51}
$ \{X^{N,i}\}_{\nN \in\N }$, $\{ \ANirs \}_{N\in\N }$ and 
$\{ \BNirsp \}_{N \in\N }$ 
are tight for each $ i , r , s ,\p \in \N $. 
\end{lem}
\begin{proof} 
The tightness of $ \{X^{N,i}\}_{\nN \in\N }$ is clear from \As{I1}. 

We note that $ \{\nablax  \aN \}_{\nN }$ is uniformly bounded 
on $ \Sr \ts \SS $ for each $ r \in \N $ by \As{H4}. 
Hence from this and \As{I1} there exists a constant $ \Ct{;51d}$ independent of $ \nN $ such that for all $ 0 \le u,v \le T $ 
\begin{align}&\notag %\label{:51a}&
\mathrm{E} ^{\muNl }
[|\ANirs (u) - \ANirs (v) |^{4}; \sup_{t\in [0,T] }|X_t^{N,i}| \le a ] 
 \le \cref{;51d} |u-v|^{2} 
.\end{align}
By \As{I1} we see that $\{\ANirs (0) \}_{\nN \in \N }$ 
 is tight. Combining these deduces the tightness of $\{ \ANirs \}_{N\in\N }$.

Recall that $ \BNirsp (t) = \bNrsp (X_t^{N,i},\XNidt ) $ and that 
$ \bNrsp $ is $ \mathcal{F}_{\rrr ,s}$-measurable by assumption. 
By construction
\begin{align}& \label{:51b}
P ^{\muNl } ( X_t^{N,j} \in \Sr \text{ for all } 1\le j \le m ,\, 
0 \le t \le T 
| \ \mathcal{L}_{r+s}^{N} \le m ) = 1 
.\end{align}

Let $ \Ct{;51a}= 
\supN \| \nabla \bbb _{r,s,\p }^{\nN } \|_{\S \ts \SS _s^{m-1} } $. 
From \eqref{:52g}, \eqref{:41U}, \eqref{:51b}, and \eqref{:40z} 
we see 
\begin{align}& \notag %\label{:51c}&
 E ^{\muNl } 
[|\BNirsp (u) - \BNirsp (v)| ^{4}; \sup_{t\in [0,T] }|X_t^{N,i}| \le a ,\ 
\mathcal{L}_{r+s}^{N} \le m 
]
\\\notag = \, & 
 E ^{\muNl } 
[|\bNrsp (X_u^{N,i},\XNidu ) - \bNrsp (X_v^{N,i},\XNidv ) 
| ^{4}; \sup_{t\in [0,T] }|X_t^{N,i}| \le a ,\ 
\mathcal{L}_{r+s}^{N} \le m 
]
\\ \notag \le \, & 
 E ^{\muNl } 
[ \sum_{j=1 }^{m} 
\cref{;51a}^4
| X_u^{N,j} - X_v^{N,j}| ^{4} ; 
 \sup_{t\in [0,T] }|X_t^{N,i}| \le a ,\ \mathcal{L}_{r+s}^{N} \le m ] 
\\ \notag 
 \le \, & \cref{;51a}^4 \cref{;40i}|u-v| ^{2} 
\quad \text{ for all $ 0 \le u, v \le T $}
.\end{align}
From this, \eqref{:40x}, and \eqref{:40p}, we deduce 
the tightness of $\{ \BNirsp \}_{N\in\N }$. 
\end{proof}
\begin{lem} \label{l:53}
$ \{((X^{N,i}, \ANirs , \BNirsp ) )_{i=1}^m\}_{N \in\N } $ is tight in 
$C([0,T], \Xi ^m )$ for each $ m , r , s , \p \in \N $. 
\end{lem}
\begin{proof}
\lref{l:53} is obvious from \lref{l:51}. 
 Indeed, the tightness of the probability measures on a countable product space follows from that of the distribution of each component. 
\end{proof}

%We set $ \mathbb{X}_{r,s,\p }^{\nN } = (\mathbb{X}\rsNi )_{i=1}^{\nN } $.
% 
Assumption \As{I1} and \lref{l:53} combined with the diagonal argument 
imply that for any subsequence of 
$ \{((X^{N,i}, \ANirs , \BNirsp ) )_{i=1}^m\}
_{N, \, \p \in\N ,\, r < s < \infty } $, 
there exists a convergent-in-law subsequence, denoted by the same symbol. 
That is, for each $ \p, s, r , m\in\N $, 
\begin{align}& \label{:54a} 
\limi{\nN }
(X^{N,i}, \ANirs , \BNirsp ) _{ i=1}^{m} = 
(X^{i}, \Airs , \Birsp ) _{ i=1}^{m} 
\quad \text{ in law}
.\end{align}
We thus assume \eqref{:54a} in the rest of this section. 

Let $ \mathbf{A}^m = (\Airs )_{i=1}^m $, 
$ \BNrspm = ( \BNirsp )_{i=1}^m $, and 
$ \mathbf{X}^{m}= (X^i)_{i=1}^m $ 
for $ \mathbf{X}=(X^i)_{i\in\N }$ in \tref{l:21}. 

\begin{lem} \label{l:54} 
For each $ m \in \N $ 
\begin{align}\label{:54b} &
 \limi{\nN } 
F (\mathbf{X}^{N,m}, \mathbf{A}^{\nN ,m} , \BNrspm ) = 
F (\mathbf{X}^{m}, \mathbf{A}^m , \Brspm ) 
\quad \text{ in law}
.\end{align}
Moreover, $ \Airs $ and $ \Birsp $ are given by 
\begin{align}\label{:54u}&
 \Airs (t) = \aaa (X_t^{i},\Xidt ) , \quad 
\Birsp (t) = \brsp (X_t^{i},\Xidt ) 
.\end{align}
\end{lem}
\begin{proof}Recall that $ F (\xi , \eta , \zeta )$ is continuous. 
Hence \eqref{:54b} follows from \eqref{:54a}. 
By \As{H4} we see $ \{\aN \}$ converges to $ \aaa $ uniformly on each 
$ \SrSS $. Then, from this, \eqref{:41r}, and \eqref{:52g} 
we obtain \eqref{:54u}. 
\end{proof}
\begin{lem} \label{l:55} For each $ m \in \N $ 
\begin{align}\notag 
& \limi{\nN } \sum_{i=1}^m \int_0^{\cdot }
\nc \cdot \sN (X_u^{N,i},\XNidu )dB _u^i = 
\mart 
\quad \text{ in law},\end{align}
where $ (\hat{B}^i)_{i=1 }^m $ is the first $ m $-components of 
 a $ (\Rd )^{\N } $-valued Brownian motion $ (\hat{B}^i)_{i\in\N } $. 
\end{lem}
\begin{proof}
By the calculation of quadratic variation, we see 
\begin{align*}&
\bra 
\int_{0}^{\cdot} 
\partial _{i,k} \hhh (\mathbf{X}_u^{N,m}) 
\sum_{n=1}^d \sN _{kn}(X_u^{N,i},\XNidu ) dB_{u}^{i,n} 
, 
\int_{0}^{\cdot} 
\partial _{j,l} \hhh (\mathbf{X}_u^{N,m}) 
\sum_{n=1}^d \sN _{ln}(X_u^{N,j},\XNjdu ) dB_{u}^{j,n}
 \ket_u 
\\ \notag &
= \delta_{ij}
\int_0^{\cdot } \aN _{kl}(X_u^{N,i},\XNidu )
\partial _{i,k} \hhh (\mathbf{X}_u^{N,m}) 
\partial _{i,l} \hhh (\mathbf{X}_u^{N,m}) du
.\end{align*}
From \As{H4}, we see that 
$\aN $ converges to $ \mathsf{a}$ uniformly on $ \Sr $ for each $ r \in \N $. 
Hence we deduce from \As{I1} and $ \hhh \in C_0^{\infty} (\S ^m)$ 
the convergence in law such that 
\begin{align} \notag %\label{:52e} 
\limi{\nN }&\sum_{i=1}^m
\int_0^{\cdot } \aN _{kl}(X_u^{N,i},\XNidu )
\partial _{i,k} \hhh (\mathbf{X}_u^{N,m}) 
\partial _{i,l} \hhh (\mathbf{X}_u^{N,m}) du
\\ \notag &
= \sum_{i=1}^m \int_0^{\cdot } \aaa _{kl}(X_u^{i},\Xidu )
\partial _{i,k} \hhh (\mathbf{X}_u^{m}) 
\partial _{i,l} \hhh (\mathbf{X}_u^{m}) du
%\quad \text{ in law}
.\end{align}
Then the right-hand side gives the quadratic variation of 
$\Mart $. This completes the proof. 
\end{proof}

\bs

We are now ready for the proof of \tref{l:21}.

\noindent 
{\em Proof of \tref{l:21}. } 
From \lref{l:56} and \eqref{:52h} we deduce that 
\begin{align}\notag %\label{:59a}&
\limsupi{\nN } &\mathrm{E} ^{\muNl }\Big[
\sup_{0\le t \le T }
\Big|F (\mathbf{X}^{N,m}, \mathbf{A}^{\nN ,m} , \BNrspm )(t) 
\\ \notag & \quad \quad 
-
\sum_{i=1}^m \int_0^{t} 
\nc \cdot \sN (X_u^{N,i},\XNidu )dB _u^i \Big|^{\phat } \Big]
\\ & \notag 
\le \limsupi{\nN } 
\mathrm{E} ^{\muNl }\big[ (\QN )^{\phat }+ (\RN )^{\phat }\big] 
=: \cref{;41}(s,\p ) 
,\end{align}
where $ 0\le \Ct{;41}(s,\p )=\cref{;41}(s,\p ,\hhh ) \le \infty $ is a constant 
depending on $ s , \p , \hhh $. 
Applying \lref{l:54} and \lref{l:55} to \eqref{:52h}, we then deduce that 
\begin{align*}& 
 \mathrm{E} ^{\mul }\big[ \sup_{0\le t \le T } \big|
F (\mathbf{X}^{m}, \mathbf{A}^m , \Brspm )(t) 
-
\sum_{i=1}^m 
\int_0^{t } \8 \cdot \sigma (X_u^{i},\Xidu )d\hat{B }_u^i 
\big|^{\phat } \big] 
\le \cref{;41}(s,\p )
.\end{align*}
From this and \eqref{:52p}, we obtain that 
\begin{align}\label{:59f}&
 \mathrm{E} ^{\mul }\big[\sup_{0\le t \le T } \big|
\hhh (\mathbf{X}_{t }^{m} ) - \hhh (\mathbf{X}_0^{m} ) - 
\sum_{i=1}^m 
\int_0^{t }\8 \cdot 
 \sigma (X_u^{i},\Xidu )d\hat{B }_u^i 
\\ \notag & \quad 
 - \sum_{i=1}^m 
 \9 
\\ \notag & \quad 
 - \sum_{i=1}^m \int_0^{t } 
\mathsf{b}_{r,s,\p}(X_u^{i},\Xidu ) \cdot 
 \8 du \big|^{\phat } \big]
\\ \notag &
\le \cref{;41}(s,\p ) 
.\end{align}

Take $ \hhh = \hhh _R \in C_0(\S ^m )$ 
such that $ \hhh (x_1,\ldots,x_m) = x_i $ for $ \{ |x_j|\le R; j=1,\ldots,m \} $ while keeping $ |\nabla_i \hhh |$ bounded in such a way that 
\begin{align*}&
\cref{;41}(\p , s )=\sup_R\cref{;41}(\p , s ,R) = o (\p , s )
.\end{align*}
Then we deduce from \eqref{:59f} that 
\begin{align}\label{:59h}&
 \mathrm{E} ^{\mul }\big[\sup_{0\le t \le T }\big|
X_{t\wedge \tau_R}^i - 
X_0^i 
- \int_0^{t\wedge \tau_R } 
\sigma (X_u^i,\mathsf{X}_u^{\diai }) 
d\hat{B }_u^i 
 \\ \notag &
- 
\int_0^{t\wedge \tau_R } 
\{\brsp (X_u^i,\mathsf{X}_u^{\diai })+ 
\btail (X_u^i)\} du 
\big|^{\phat } \big]
\le \cref{;41}(s,\p ) 
,\end{align}
where $ \tau_R $ is a stopping time such that, for 
$ \mathbf{X}^m=(X^{i},\Xid )_{i=1}^m \in C([0,T];(\S \ts \SS )^m)$, 
\begin{align} \notag %\label{:59g}
& 
\tau_R = \inf\{ t>0; |X_t^i| \ge R \text{ for some } i=1,\ldots,m \} 
.\end{align}
As $ R > 0 $ is arbitrary, \eqref{:59h} holds for all $ R > 0$. 
Taking $ R \to \infty $, we thus obtain 
\begin{align}\label{:59i}&
\mathrm{E} ^{\mul }\big[\sup_{0\le t \le T } \big|
X_{t}^i - 
X_0^i 
- \int_0^{t } 
\sigma (X_u^i,\mathsf{X}_u^{\diai }) 
d\hat{B }_u^i 
 \\ \notag &\quad \quad \quad \quad \quad \quad \quad 
- 
\int_0^{t } 
\{\brsp (X_u^i,\mathsf{X}_u^{\diai })+ 
\btail (X_u^i)\} du 
\big| \big] 
\\ \notag \le \, & 
\liminfi{R} \mathrm{E} ^{\mul }\big[\sup_{0\le t \le T }\big|
X_{t\wedge \tau_R}^i - 
X_0^i 
- \int_0^{t\wedge \tau_R} 
\sigma (X_u^i,\mathsf{X}_u^{\diai }) 
d\hat{B }_u^i 
 \\ \notag & \quad \quad \quad \quad \quad \quad \quad 
- 
\int_0^{t\wedge \tau_R} 
\{\brsp (X_u^i,\mathsf{X}_u^{\diai })+ 
\btail (X_u^i)\} du 
\big| \big]
\\ \notag 
\le \, &\cref{;41}(s,\p )^{1/{\phat }} 
.\end{align}
We note here that the integrands in the first and second lines of \eqref{:59i} 
are uniformly integrable because of \eqref{:59h}. 
Taking $ \p \to \infty $, then $ s \to \infty $ in \eqref{:59i}, and using assumptions \eqref{:41tt} and \eqref{:41z} we thus obtain 
\begin{align}& \notag 
 \mathrm{E} ^{\mul }\big[\sup_{0\le t \le T } \big| 
X_{t}^i - X_0^i 
- \int_0^{t} \sigma (X_u^i,\mathsf{X}_u^{\diai }) d\hat{B }_u^i 
% \\ \notag &
- 
\int_0^{t} 
\{\bbb (X_u^i,\mathsf{X}_u^{\diai })+ 
\btail (X_u^i)\} du 
\big| \big] = 0 
.\end{align}
This implies for all $ 0\le t \le T $ 
\begin{align}\label{:59l}&
X_{t}^i - X_0^i 
- \int_0^{t} \sigma (X_u^i,\mathsf{X}_u^{\diai }) 
d\hat{B }_u^i 
- 
\int_0^{t} 
\{\bbb (X_u^i,\mathsf{X}_u^{\diai })+ 
\btail (X_u^i)\} du = 0 
.\end{align}
We deduce \eqref{:41b} from \eqref{:59l}, which completes the proof of \tref{l:21}. 
\qed 

\section{Proof of \tref{l:22} } \label{s:4}

Is this section we prove \tref{l:22} using \tref{l:21}. 
\As{H1}--\As{H4} are commonly assumed in \tref{l:22} and \tref{l:21}. 
Hence our task is to derive condition \0 from conditions stated in \tref{l:22}. 
From \As{J2} we easily deduce that 
\begin{align} \label{:60a}&
\limi{N}\uN =\uu \quad \text{ in } L^{\phat }_{\mathrm{loc}}(\S ,dx)
,\\ \label{:60b}&
\limi{N}\ggNs = \ggs \quad \text{ in }L^{\phat }_{\mathrm{loc}}(\muone ) 
\quad \text{for all }s 
.\end{align}

\begin{lem} \label{l:61}
$ \mu $ has a logarithmic derivative $ \dmu $ in $ \Lplochat $, 
where $ 1 \le p < \phat $. 
\end{lem}
\begin{proof}
We use a general theory developed in \cite{o.isde}. 
\As{H1} corresponds to \thetag{4.1} and \thetag{4.2} in \cite{o.isde}. 
\eqref{:60a}, \eqref{:60b}, \eqref{:21m}, and \eqref{:21r} 
correspond to \thetag{4.15}, \thetag{4.30}, \thetag{4.29}, 
and \thetag{4.31} in \cite{o.isde}. 
Then all the assumptions of \cite[Theorem 45]{o.isde} are satisfied. 
We thus deduce \lref{l:61} from \cite[Theorem 45]{o.isde}. 
\end{proof}

\bs 

Let $ \{ \mathbf{X}^{\nN } \}_{\nN \in \N } $ be a sequence of solutions in \eqref{:20m} and \eqref{:20n}. We set the $ m $-labeling 
\begin{align}\label{:62p}&
 \mathbf{X}^{\nN , [m] } = 
(X^{\nN , 1},\ldots, X^{\nN , m},\sum_{j= 1+m }^{\nN } \delta_{X^{\nN , j}}) 
.\end{align}
It is known \cite{o.tp, o.isde} that $ \mathbf{X}^{\nN , [m] } $ 
is a diffusion process associated with the Dirichlet form 
$ \mathcal{E}^{\muNm } $ on $ L^2(\S ^{m}\ts \SS , \muNm ) $ such that 
\begin{align}\label{:62q}&
\mathcal{E}^{\muNm } (f,g) = 
\int_{\S ^{m}\ts \SS } 
\half \{ \sum_{i=1}^m \nabla _i{f} \cdot \nabla _i{g} \} + 
\mathbb{D}[f,g] d\muNm 
,\end{align}
where the domain $ \mathcal{D}^{[m]}$ is taken as the closure of 
$ \mathcal{D}_0^{[m]}= C_0^{\infty} (\S ^m)\ot \mathcal{D}_{\circ } $. 
Note that the coordinate function $ x_i = x_i\ot 1 $ is locally in $ \mathcal{D}^{[m]}$. 
From this we can regard $ \{X _t^{N,i}\}$ as a Dirichlet process of the 
$ m $-labeled diffusion $ \mathbf{X}^{\nN } $ 
associated with the Dirichlet space as above. 
In other words, we can write 
\begin{align}&\notag %\label{:61a}&
X _t^{N,i} - X _0^{N,i} = 
 f_i (\mathbf{X}^{\nN } _t) - f_i (\mathbf{X}^{\nN } _0) 
=: A _t^{[f_i]} 
,\end{align}
where $ f_i (\mathbf{x},\mathsf{s}) = x_i \ot 1 $, $ x_i \in \Rd $, and 
$ \mathbf{x} = (x_j)_{j=1}^m \in (\Rd )^m$. 
By the Fukushima decomposition of $ X_t^{N,i} $, there exist a unique continuous local martingale additive functional $ \MNi = \{ \MNi _t \} $ and 
an additive functional of zero energy 
$ \mathsf{N}^{\nN ,i } = \{ \mathsf{N}^{\nN ,i } _t \} $ such that 
\begin{align}&\notag %\label{:61b}&
 X_t^{N,i} - X_0^{N,i} = \MNi _t + \mathsf{N}^{\nN ,i } _t 
.\end{align}
We refer to \cite[Chapter 5]{FOT.2} for the Fukushima decomposition. 
Because of \eqref{:20m}, we then have 
\begin{align}&\notag %\label{:61c}&
 \MNi _t = \int_0^t \sN (X_u^{N,i},\XNidu )dB_u^i 
,\quad 
 \mathsf{N}^{\nN ,i } _t = \int_0^t \bbb ^{\nN }(X_u^{N,i},\XNidu )du 
.\end{align}
\begin{lem} \label{l:62} 
Let $ \map{r_T}{C([0,T];\S )}{C([0,T];\S )}$ be such that 
$ r_T (X )_t = X_{T-t}$. 
Suppose that $ \mathbf{X}_0^{\nN , [m] } = \mu ^{N,[m]} $ in law. 
Then 
\begin{align} \label{:62a}
X_t^{N,i} - X_0^{N,i} 
&=\half \MNi _t + \half (\MNi _{T-t} (r_{T}) - \MNi _{T} (r_{T})) 
\quad \text{ a.s.}
.\end{align}
\end{lem}
\begin{proof} 
Applying the Lyons-Zheng decomposition \cite[Theorem 5.7.1]{FOT.2} to 
additive functionals $ A ^{[f_i]} $ for $1 \le i\le m$, 
we obtain \eqref{:62a}. 
\end{proof}

\begin{lem} \label{l:63} 
\As{I1} holds. 
\end{lem}
\begin{proof} 
Although $ \MNi $ is a $ d $-dimensional martingale by definition, 
we assume $ d =1 $ here and prove only this case for simplicity. 
The general case $ d \ge 1 $ can be proved in a similar fashion. 
Let $ \cref{;3}$ be the constant in \eqref{:20o} (under the assumption $ d=1$). Then we note that for $ u \ge v$
\begin{align}\label{:63z}&
0 \le \langle \MNi \rangle _u - \langle \MNi \rangle _v = 
 \int_v^u \ANirs (t) dt \le \cref{;3} (u-v)
\end{align}

We begin by proving \eqref{:40z}. 
From a standard calculation of martingales and \eqref{:63z}, we obtain 
\begin{align}\notag %\label{:63a}
\mathrm{E} ^{\muNl }[|\MNi _u - \MNi _v |^{4}] &=
\mathrm{E} ^{\muNl }
[|B _{\langle \MNi \rangle _u} - B _{\langle \MNi \rangle _v }|^{4 }] 
\\&\notag 
= 3\mathrm{E} ^{\muNl }
\bigl[|\langle \MNi \rangle _u - \langle \MNi \rangle _v |^{2 }\bigr] 
%\\& \notag = 3\mathrm{E} ^{\muNl }\bigl[ | \int_v^u \ANirs (t) dt|^2\bigr] 
\\& \notag 
\le \cref{;53} |u-v|^2
,\end{align}
where $ \Ct{;53}=3\cref{;3}^2 $ and $ \{B_t \}$ 
is a one-dimensional Brownian motion. 
%Here we used \As{H4} and the $ \muNl $-symmetry of the $ m $-labeled diffusion $ \mathbf{X}^{\nN , [m] } $ in the last line. 
Applying the same calculation to 
$ \MNi _{T-t} (r_{T}) - \MNi _{T} (r_{T}) $, we have 
\begin{align}\label{:63b} & \quad \quad \quad 
\mathrm{E} ^{\muNl }[|\MNi _{T-t} (r_{T}) - \MNi _{T-u} (r_{T})|^{4 }]
 \le \cref{;53} |t-u|^{2} 
\quad \text{ for each $ 0 \le t,u \le T $}
.\end{align}
Combining \eqref{:62a} and \eqref{:63b} with the Lyons-Zheng decomposition 
\eqref{:62a}, we thus obtain 
\begin{align}&\label{:63c} 
\quad \quad \quad 
\mathrm{E} ^{\muNl }[|X _t^{N,i} - X _u^{N,i}|^{4 }] \le 
2\cref{;53} |t-u|^{2} 
\quad \text{ for each $ 0 \le t,u \le T $}
.\end{align}
Taking a sum over $ i=1,\ldots,m$ in \eqref{:63c}, we deduce \eqref{:40z}. 
%In fact, we obtain a stronger inequality than \eqref{:40z}. 

We next prove \eqref{:40x}. From \eqref{:62a} we have 
\begin{align} \notag 
2 |X_t^{N,i} - X_0^{N,i} | 
& \le |\MNi _t | + | \MNi _{T-t} (r_{T}) - \MNi _{T} (r_{T}) | 
\quad \text{ a.s.}
.\end{align}
From this and a representation theorem of martingales, we obtain 
\begin{align}\label{:63e} & 
P ^{\muNl } (\sup_{t\in[0,T]} |X _t^{N,i} - X_0^{N,i} |\ge a ) 
\\ \notag 
\le & 
P ^{\muNl } (\sup_{t\in[0,T]} |\MNi _t |\ge a ) + 
P ^{\muNl } 
(\sup_{t\in[0,T]} |\MNi _{T-t} (r_{T}) - \MNi _{T} (r_{T}) |\ge a ) 
\\ \notag 
= & 
2 P ^{\muNl } (\sup_{t\in[0,T]} |\MNi _t |\ge a ) 
\\ \notag 
= & 
2 P ^{\muNl } (\sup_{t\in[0,T]} |B_{\langle \MNi \rangle _t }|\ge a ) 
.\end{align}
A direct calculation shows
\begin{align}\label{:63f}&
P ^{\muNl } (\sup_{t\in[0,T]} |B_{\langle \MNi \rangle _t }|\ge a ) 
\le 
P ^{\muNl } (\sup_{t\in[0,\sqrt{\cref{;3}}T]}
 |B_{t}|\ge a ) \le 
\mathrm{Erf} (\frac{a}{\sqrt{\cref{;3}}T}) 
\end{align}
From \eqref{:63e}, \eqref{:63f}, and \As{H2}, we obtain \eqref{:40x}. 

We proceed with the proof of \eqref{:40p}. 
Similarly as \eqref{:63e} and \eqref{:63f}, we deduce 
\begin{align}\label{:63h} 
P ^{\muNl } (\inf_{t\in[0,T]} |X _t^{N,i} |\le r ) 
%\\ \notag &
\le &
P ^{\muNl } 
(\sup_{t\in[0,T]} |X _t^{N,i} - X_0^{N,i} |\ge |X_0^{N,i} |-r ) 
\\ \notag \le &
2 P ^{\muNl } (\sup_{t\in[0,T]} |\MNi _t |\ge |X_0^{N,i} |-r ) 
\\ \notag \le & 
2 \int_{\SS } \mathrm{Erf} (\frac{|s_i|-r}{\sqrt{\cref{;3}}T}) 
\muN (d\mathsf{s})
,\end{align}
where $ s_i = \labN (\mathsf{s})_i$. 
We note that $ X_0^{N,i} =s_i $ by construction. 
From \eqref{:63h} and \eqref{:21T}, we deduce 
\begin{align}\notag %\label{:63j}
 \limsupi{\nN } P ^{\muNl } (\mathcal{L}_r^{N} > L ) 
\le \, & \limsupi{\nN }
 \sum_{i>L} P ^{\muNl } (\inf_{t\in[0,T]} |X _t^{N,i} |\le r ) 
\\ \notag \le \, & 2
 \limsupi{\nN } \sum_{i>L} 
\int_{\SS } \mathrm{Erf} (\frac{|s_i|-r}{\sqrt{\cref{;3}}T}) 
\muN (d\mathsf{s}) 
\\ \notag \to &\, 0 \quad (L \to \infty )
.\end{align} 
This completes the proof. 
\end{proof}

\begin{lem} \label{l:64} 
\As{I2} holds. 
\end{lem}
\begin{proof}
\eqref{:40h} follows from \eqref{:60a}, \eqref{:60b}, and \eqref{:21r}. 
For each $ i \in \N $ we deduce that 
\begin{align}\label{:64a}
 E ^{\muNl } 
[\int_0^T | \bNrsp (X_t^{N,i},\XNidt ) | ^{\phat } dt ] 
\le & \, 
\sum_{i=1}^{\nN } E ^{\muNl } 
[\int_0^T | \bNrsp (X_t^{N,i},\XNidt ) | ^{\phat } dt ] 
\\ \notag = & \, 
E ^{\muNl } [ \sum_{i=1}^{\nN }
\int_0^T | \bNrsp (X_t^{N,i},\XNidt ) | ^{\phat } dt ] 
\\ \notag = &\, 
E ^{\muNone } 
[\int_0^T | \bNrsp (\mathbf{X}_t^{\nN , [1] } ) | ^{\phat } dt 
] 
.\end{align}
Diffusion process $ \mathbf{X}^{\nN , [1] } $ in \eqref{:62p} with $ m = 1 $ 
given by the Dirichlet form $ \mathcal{E}^{\muNone } $ in 
\eqref{:62q} is $ \muNone $-symmetric. 
Hence we see that for all $ 0 \le t \le T $ 
\begin{align} \notag 
 E ^{\muNone } 
[| \bNrsp (\mathbf{X}_t^{\nN , [1] } ) | ^{\phat } 
] 
\le \,
\int_{\S \ts \SS } |\bNrsp | ^{\phat } d\muNone 
.\end{align}
This yields 
\begin{align}\label{:64b}&
\int_0^T dt \, E ^{\muNone } 
[| \bNrsp (\mathbf{X}_t^{\nN , [1] } ) | ^{\phat } 
] 
\le \, T 
\int_{\S \ts \SS } |\bNrsp | ^{\phat } d\muNone 
.\end{align}
From \eqref{:64a} and \eqref{:64b} we obtain \eqref{:40i}. 
\end{proof}

\begin{lem} \label{l:65} 
\As{I3}--\As{I5} hold. 
\end{lem}
\begin{proof} Conditions \eqref{:41r}, \eqref{:41U}, and \eqref{:41s} 
follow from \As{J1}, \As{J2}, \As{I1}, \As{I2}, and \eqref{:21l}. 
Similarly, as \lref{l:64}, we obtain for each $ i \in \N $ 
\begin{align}\label{:65a}&
 E ^{\muNl } 
[\int_0^T | (\7 ) (X_t^{N,i},\XNidt ) | ^{\phat } dt ] 
\le T \int_{\S \ts \SS } |\7 | ^{\phat } d\muNone 
.\end{align}
Hence \eqref{:41t} follows from \eqref{:65a} and \eqref{:41s}. 
\eqref{:41tt} follows from \eqref{:50a} and an inequality similar to \eqref{:65a}. 
We have thus obtained \As{I3}. 
Condition 
\eqref{:41u} follows from \As{J1} and \As{J2}. 
Similarly, as \lref{l:64}, we obtain for each $ i \in \N $ 
\begin{align}&\notag %\label{:66a}&
 E ^{\muNl } 
[\int_0^T | ( \bNrstail - \btail ) (X_t^{N,i},\XNidt ) | ^{\phat } dt ] 
\le T \int_{\S \ts \SS } | \bNrstail - \btail | ^{\phat } d\muNone 
.\end{align}
This together with \eqref{:41u} implies \eqref{:41v}. Hence we have \As{I4}. 
Similarly as \lref{l:64}, we obtain \eqref{:41z} from \eqref{:41a}. 
We have thus obtained \As{I5}. 
\end{proof}

\bs
\noindent {\em Proof of \tref{l:22}. } 
\0 follows from \lref{l:63}--\lref{l:65}. 
Hence we deduce \tref{l:22} from \tref{l:21}. 
\qed

We finally present a sufficient condition of \eqref{:21T}. 
\begin{lem} \label{l:67}
Assume \As{H1} and \eqref{:25f} for each $ r \in \N $ 
as \sref{s:2}. We take the label $ \labN $ as \eqref{:25g}. 
Then 
\eqref{:21T} holds. 
\end{lem}
\begin{proof} 
Let 
 $ \Ct{;67} = \cref{;67}(\nN )$ be such that 
 \begin{align}\notag %\label{:67d}
 \cref{;67} &=\, 
\int_{\S }
\mathrm{Erf} (\frac{|x|-r}{\sqrt{\cref{;3}}T}) \rho^{\nN , 1} (x) dx 
.\end{align}
Let $ \Ct{;68}=\limsupi{N}\cref{;67}(\nN )$. 
Then from \As{H1} and \eqref{:25f}, we see that for each large 
$ r $
\begin{align}\label{:67x}
\cref{;68}& \le \limi{\nN } 
\int_{\Sr } 
\mathrm{Erf} (\frac{|x|-r}{\sqrt{\cref{;3}}T}) \rho^{\nN , 1} (x) dx 
+ 
\limsupi{\nN } 
\int_{\S \backslash \Sr } 
\mathrm{Erf} (\frac{|x|-r}{\sqrt{\cref{;3}}T}) \rho^{\nN , 1} (x) dx 
\\ \notag &
< \infty 
.\end{align}

From \As{H1} we see that $ \{\muN \}_{\nN \in \N }$
 converges to $ \mu $ weakly. 
Hence $ \{\muN \}_{\nN \in \N }$ is tight. 
This implies that there exists a sequence of increasing sequences of 
natural numbers $ \mathbf{a}_n = \{ a_n(m) \}_{m=1}^{\infty} $ 
such that $ \mathbf{a}_n < \mathbf{a}_{n +1}$ 
and that for each $ m $ 
\begin{align*}&
\limi{n} \limsupi{\nN } \muN (\mathsf{s} (\S _{m}) \ge a_n(m) ) = 0 
. \end{align*}
Without loss of generality, we can take 
$ a_n(m) > m $ for all $ m , n \in \N $. 
Then from this, we see that there exists a sequence 
$ \{\mathsf{p}(L) \}_{L\in\N }$ converging to $ \infty $ 
 such that $ \mathsf{p}(L) < L $ for all $ L \in \N $ and that 
\begin{align} \label{:67c}&
\limi{L}\limsupi{N} \muN (\mathsf{s}(\S _{\mathsf{p}(L)}) \ge L ) = 0 
.\end{align}

Recall that the label $ \labN (\mathsf{s}) = (s_i)_{i\in\N } $ 
 satisfies $ |s_1|\le |s_2|\le \cdots $. 
Using this, we divide the set $ \SS $ as in such a way that 
\begin{align*}&
\text{$ \{ s_L \in \S _{\mathsf{p}(L)} \} $ and 
$ \{ s_L \not\in \S _{\mathsf{p}(L)} \} $.}
\end{align*}
Then $ \mathsf{s} \in \{ s_L \in \S _{\mathsf{p}(L)} \} $ if and only if 
$ \mathsf{s}(\S _{\mathsf{p}(L)}) \ge L $. Hence 
we easily see that 
\begin{align}\notag &%\label{:67f} &
\sum_{i>L} 
\int_{\SS } \mathrm{Erf} (\frac{|s_i|-r}{\sqrt{\cref{;3}}T}) 
\muN (d\mathsf{s}) 
%\\ \notag 
\le 
\cref{;67}(\nN )
 \muN (\{ \mathsf{s} (\S _{\mathsf{p}(L)}) \ge L \} )
+
\int_{\S \backslash \S _{\mathsf{p}(L)}} 
\mathrm{Erf} (\frac{|x|-r}{\sqrt{\cref{;3}}T}) \rho^{\nN , 1} (x) dx
.\end{align}
Taking the limits on both sides, we obtain 
\begin{align}& \notag %\label{:67g}
\limi{L}\limsupi{N}
\sum_{i>L} 
\int_{\SS } \mathrm{Erf} (\frac{|s_i|-r}{\sqrt{\cref{;3}}T}) 
\muN (d\mathsf{s}) 
\le
\\\notag 
 \, \cref{;68}&
 \limi{L}\limsupi{N}
\muN (\{ \mathsf{s} (\S _{\mathsf{p}(L)}) \ge L \} ) 
 + 
 \limi{L}\limsupi{N}
\int_{\S \backslash \S _{\mathsf{p}(L)}} 
\mathrm{Erf} (\frac{|x|-r}{\sqrt{\cref{;3}}T}) \rho^{\nN , 1} (x) dx
.\end{align}
Applying \eqref{:67x} and \eqref{:67c} to the second term, and 
\eqref{:25f} to the third, we deduce \eqref{:21T}. 
\end{proof}

\section{Examples}\label{s:5}
The finite-particle approximation in \tref{l:22} contains many examples 
 such as Airy$ _{\beta }$ point processes ($ \beta = 1,2,4$), Bessel$ _{2,\alpha }$ point process, the Ginibre point process, the Lennard--Jones 6-12 potential, and Riesz potentials.The first three examples are related to random matrix theory and the interaction 
$ \Psi (x) = - \log |x| $, the logarithmic function. 
 We present these in this section. For this we shall confirm the assumptions in \tref{l:22}, 
that is, assumptions \As{H1}--\As{H4} and \As{J1}--\As{J6}.

Assumption \As{H1} is satisfied for the first three examples \cite{mehta,Sos00}. 
As for the last two examples, we assume \As{H1}. We also assume \As{H2}. 
\As{H3} can be proved in the same way as given in \cite{o-t.tail}. 
In all examples, $ \aaa $ is always a unit matrix. Hence it holds that 
\As{H4} is satisfied and that 
\eqref{:21j} in \As{J1} becomes $ \bbb ^{\nN } = \half \dmuN $. 
From this we see that SDEs \eqref{:21z} and \eqref{:21w} become 
\begin{align}\label{:90b}
dX_t^{N,i}=\, & dB_t^{N,i}+\half \dmuN (X_t^{N,i}, \XNidt)\,dt\quad (1\le i\le N)
,\\\label{:90a}
dX_t^i=\, & dB_t^i+\half \dmu (X_t^i, \Xidt)\,dt\quad (i\in\N)
,\end{align}
where $ \dmu $ is the logarithmic derivative of $ \mu $ given by \eqref{:21v}. 
Assumption \As{J6} for the first three examples with $ \beta = 2 $ 
can be proved in the same way as \cite{o-t.tail} as we explained in \rref{r:K5}. 
Thus, in the rest of this section, our task is to check assumptions \As{J2}--\As{J5}.

\subsection{The Airy$ _{\beta }$ interacting Brownian motion ($ \beta = 1,2,4$)} 
\label{s:52} 
Let $ \mu _{\mathrm{Airy}, \beta } ^{\nN }$ and 
$\mu _{\mathrm{Airy},\beta }$ be as in \sref{s:1}. 
Recall SDEs \eqref{:10l} and \eqref{:10H} in \sref{s:1}. 
Let $ \mathbf{X}^{\nN } = (X^{\nN ,i})_{i=1}^{\nN }$ and 
$ \mathbf{X}=(X^i)_{i\in\N }$ be solutions of 
\begin{align} \tag{\ref{:10l}} 
dX_t^{\n ,i} = \,& dB_t^i + 
\frac{\beta }{2} \sum_{j=1, \,  j\not= i}^{\n } 
\frac{1}{X_t^{\n ,i} - X_t^{\n ,j} } dt 
- \frac{\beta }{2 } 
\{ \n ^{1/3} + \frac{1}{2\n ^{1/3}}X_t^{\n ,i} \}dt 
,\\%\end{align}
%\intertext{and let $ \mathbf{X}=(X^i)_{i\in\N }$ be a solution of }%\begin{align}
\tag{\ref{:10H}} 
dX_t^i= \, & dB_t^i+
\frac{\beta }{2}
\limi{r} \{ \sum_{|X_t^j|<r, j\neq i} \frac{1}{X_t^i-X_t^j} -\int_{|x|<r}\frac{\varrho (x)}{-x}\,dx\}dt \quad (i\in\N)
.\end{align}
\begin{prop}\label{l:32}
If $ \beta = 1,4$, then each sub-sequential limit of solutions $  \mathbf{X}^{\nN } $ 
of \eqref{:10l} 
satisfies \eqref{:10H}. 
If $ \beta =2$, then the full sequence converges to \eqref{:10H}. 
\end{prop}
\begin{proof} 
Conditions  \As{J2}--\As{J5} other than \eqref{:21q} can be proved in the same way 
as given in \cite{o-t.airy}. In \cite{o-t.airy}, we take $ \chi_s (x) = 1_{\S _s} (x)$; its 
adaptation to the present case is easy. 

We consider estimates of correlation functions such that 
\begin{align}\label{:92aa}&
\inf_{\nN \in \N }\rho_{\mathrm{Airy}, \beta }  ^{\nN ,1} (x) 
\ge 
\cref{;32a} \quad \text{ for all } x \in \Sr 
\\\label{:92a}& 
\sup_{\nN \in \N }\rho_{\mathrm{Airy}, \beta }  ^{\nN ,2} (x,y)
 \le 
\cref{;32} |x-y| \quad \text{ for all } x,y \in \Sr 
,\end{align}
where $ \Ct{;32a}(r)$ and $ \Ct{;32}(r) $ are positive constants. 
The first estimate is trivial because $ \rho_{\mathrm{Airy}, \beta }  ^{\nN ,1} $ 
converges to $ \rho_{\mathrm{Airy}, \beta }  ^{1} $ uniformly on $ \Sr $ and, 
all these correlation functions are continuous and positive. 
The second estimate follows 
from the determinantal expression of the 
correlation functions and bounds on derivative of determinantal kernels. 
Estimates needed for the proof can be found in \cite{o-t.airy} and 
the detail of the proof of \eqref{:92a} is left to the reader.

Equation \eqref{:21q} follows from \eqref{:92aa} and \eqref{:92a}. 
Indeed, the integral in \eqref{:21q} is taken on the bounded domain and 
the singularity of integral of $ \gN (x,y)=\beta /(x-y)$ near $ \{x=y\}$ is logarithmic. 
Furthermore,  
the one-point correlation function $ \rho_{\mathrm{Airy}, \beta, x  }  ^{\nN ,1}$ 
of the reduced Palm measure conditioned at $ x$ 
is controlled by the upper bound of the two-point correlation function and 
the lower bound of one-point correlation function because
\begin{align*}&
\rho_{\mathrm{Airy}, \beta, x  }  ^{\nN ,1} (y) =
 \frac{\rho_{\mathrm{Airy}, \beta }  ^{\nN ,2} (x,y)}
{ \rho_{\mathrm{Airy}, \beta }  ^{\nN ,1} (x)}
.\end{align*}
Using these facts, we see that \eqref{:92aa} and \eqref{:92a} imply \eqref{:21q}. 
\end{proof}

\subsection{The Bessel$ _{2,\alpha }$ interacting Brownian motion} \label{s:53}
Let $\S =[0,\infty)$ and $\alpha \in[1,\infty)$. 
We consider the Bessel$ _{2,\alpha }$ point process 
$ \mu _{\mathrm{bes},2, \alpha }$ and their $ \nN $-particle version. 
The Bessel$ _{2,\alpha }$ point process $ \mu _{\mathrm{bes},2,\alpha }$ is 
a determinantal point process with kernel 
\begin{align}\label{:93p}
\mathsf{K} _{\mathrm{bes},2,\alpha } (x,y) & = 
 \frac{J_{\alpha } (\sqrt{x}) \sqrt{y} J_{\alpha }' (\sqrt{y}) - 
 \sqrt{x} J_{\alpha }' (\sqrt{x})  J_{\alpha }(\sqrt{y})
 }{2(x-y)}
\\ \notag & = 
\frac{\sqrt{x} J_{\alpha +1} (\sqrt{x}) J_{\alpha } (\sqrt{y}) - 
 J_{\alpha } (\sqrt{x}) \sqrt{y} J_{\alpha +1}(\sqrt{y})
 }{2(x-y)}
,\end{align}
where $ J_{\alpha }$ is the Bessel function of order $ \alpha $ \cite{Sos00,o-h.bes}. 
The density $ \mathsf{m} _{\alpha }^{\n }(\mathbf{x}) d\mathbf{x}$ 
of the associated $ \nN $-particle systems $ \mu _{\mathrm{bes},2,\alpha }^{\nN }$ 
is given by 
\begin{align}\label{:13}& 
\mathsf{m} _{\alpha }^{\n} (\mathbf{x}) =
\frac{1}{\mathcal{Z} _{\alpha }^{\n }}
 e^{-\sum_{i=1}^{\n }x_i/4\n  } \prod_{j=1}^{\n }x_j^{\alpha }
\prod_{k<l}^{\n } |x_k-x_l|^{2 } 
%\prod_{m=1}^{\n } dx_m 
.\end{align}
It is known that $ \mu _{\mathrm{bes},2,\alpha }^{\nN }$  is also determinantal \cite[945p]{Sos00} and \cite[91p]{forrester}
The Bessel$ _{2,\alpha }$ interacting Brownian motion is 
given by the following \cite{o-h.bes}.
\begin{align}\label{:93a}
dX_t^{N,i} = &dB_t^i+ \{-\frac{1}{8N}+ 
\frac{\alpha }{2X_t^{N,i}} + 
\sum_{j=1, j\neq i}^{\nN } 
 \frac{1}{X_t^{N,i}-X_t^{N,j}}\} dt \quad (1\le i\le N)
,\\
\label{:93b}
dX_t^i= &dB_t^i+  
\{ \frac{\alpha }{2X_t^i} + 
\sum_{ j\neq i}^\infty \frac{1}{X_t^i-X_t^j} \} dt \quad (i\in\N)
.\end{align}
This appears at the hard edge of one-dimensional systems. 
\begin{prop} Assume $ \alpha > 1 $. 
Then \eqref{:21a} holds for \eqref{:93a} and \eqref{:93b}. 
\end{prop}
\begin{proof}
\As{J2}--\As{J5} except \eqref{:21T} are proved in \cite{o-h.bes}. 
We easily see that the assumptions of \lref{l:67} hold and yield \eqref{:21T}. 
We thus obtain \As{J5}. 
\end{proof}
\begin{rem}\label{r:33} 
There exist other natural ISDEs and 
$ \nN $-particle systems related to the Bessel point processes. 
They are the non-colliding square Bessel processes and their square root. 
The non-colliding square Bessel processes are reversible to the Bessel$ _{2,\alpha }$ point processes, but  the associated Dirichlet forms are different from the Bessel$ _{2,\alpha }$ interacting Brownian motion. 
Indeed, the coefficients 
$ \mathsf{a}^{\nN } $ and $  \mathsf{a} $ in \sref{s:2} are taken to be 
$  \mathsf{a}^{\nN }(x.\mathsf{y}) =  \mathsf{a}(x.\mathsf{y})= 4x $.  
On the other hand, each square root of the non-colliding Bessel processes 
is not reversible to 
the Bessel$ _{2,\alpha }$ point processes, but  has the same type of Dirichlet forms as 
the Bessel$ _{2,\alpha }$ interacting Brownian motion. In particular, the coefficients 
$ \mathsf{a}^{\nN } $ and $  \mathsf{a} $ in \sref{s:2} are taken to be 
$  \mathsf{a}^{\nN }(x.\mathsf{y}) =  \mathsf{a}(x.\mathsf{y})= 1 $.  
That is, they are constant time change of 
distorted Brownian motion with the standard square field.

We refer to \cite{KT11,KT11-b,o-t.sm} for these processes. 
For reader's convenience we provide an ISDE describing the non-colliding square Bessel processes and their square root. We note that SDE \eqref{:93d} is 
a constant time change of that in \cite{KT11-b,o-t.sm}. 
Let $ \mathbf{Y}^{N} = (Y^{N,i})_{i=1}^N $ and 
$ \mathbf{Y} = (Y^i)_{i\in\N } $ be the non-colliding square Bessel processes. 
Then for $ 1\le i\le N$ %we obtain 
\begin{align}\label{:93c}
dY_t^{N,i} = & 2 \sqrt{Y_t^{N,i}}dB_t^i+
4 \{-\frac{Y_t^{N,i}} {8N}+ \frac{\alpha +1}{2} + 
\sum_{j=1, j\neq i}^{\nN } 
\frac{ Y_t^{N,i} }{Y_t^{N,i}-Y_t^{N,j}}\} dt 
,\\
\label{:93d}
dY_t^i= & 2\sqrt{Y_t^{i}} dB_t^i+  
4 \{  \frac{\alpha +1}{2} + 
\sum_{ j\neq i}^\infty \frac{ Y_t^{i}}{Y_t^i-Y_t^j} \} dt \quad (i\in\N)
.\end{align}
Let $ \mathbf{Z}^{N} = (Z^{N,i})_{i=1}^N $ and $ \mathbf{Z} = (Z^i)_{i\in\N } $ be square root of the non-colliding square Bessel processes. Then applying It$ \hat{\mathrm{o}}$ formula we obtain from \eqref{:93c} and \eqref{:93d} 
\begin{align}\label{:93e}
dZ_t^{N,i} = & dB_t^i+ \{-\frac{Z_t^{N,i}}{ 4 N}+ 
\frac{\alpha + \half }{Z_t^{N,i}} + 
\sum_{j=1, j\neq i}^{\nN } 
 \frac{2 Z_t^{N,i}} {(Z_t^{N,i})^2-(Z_t^{N,j})^2}\} dt \ (1\le i\le N)
,\\
\label{:93f}
dZ_t^i= & dB_t^i+  
\{ \frac{\alpha + \half }{ Z_t^i} +
 \sum_{ j\neq i}^\infty \frac{ 2 Z_t^{N,i}}{(Z_t^i)^2-(Z_t^j)^2} \} dt \quad (i\in\N)
.\end{align}

We remark that \tref{l:22} can be applied to the non-colliding square Bessel processes because the equilibrium states are the same as the Bessel interacting Brownian motion and 
coefficients are well-behaved as 
$  \mathsf{a}^{\nN }(x.\mathsf{y}) =  \mathsf{a}(x.\mathsf{y})= 4x $. 
\end{rem}

\subsection{The Ginibre interacting Brownian motion} \label{s:54}
Let $\S =\R ^2$. 
Let $  \mu _{\mathrm{gin}}^{\nN }  $ and $  \mu _{\mathrm{gin}} $ 
be as in \sref{s:1}. 
Let $ \Phi ^{\nN } = |x|^2$ and $\Psi (x)=-\log |x|$. 
Then the $ \nN $-particle systems are given by 
\begin{align}\tag{\ref{:10r}}
dX_t^{N,i} = & dB_t^i-X_t^{N,i}dt + 
\sum_{j=1, j\neq i}^{\nN } 
\frac{X_t^{N,i}-X_t^{N,j}}{|X_t^{N,i}-X_t^{N,j}|^2}dt 
\quad (1\le i\le N)
.\end{align}
The limit ISDEs are
\begin{align}\tag{\ref{:10s}}
dX_t^i= & 
dB_t^i+\limi{r} \sum_{|X_t^i-X_t^j|<r, j\neq i} 
\frac{X_t^i-X_t^j}{|X_t^i-X_t^j|^2}dt \quad (i\in\N)
\\\intertext{and } \tag{\ref{:10t}}
dX_t^i=
& 
dB_t^i - X_t^idt + 
\limi{r} \sum_{|X_t^j|<r, j\neq i} \frac{X_t^i-X_t^j}{|X_t^i-X_t^j|^2}dt \quad (i\in\N)
.\end{align}
\begin{prop}\label{l:33} 
\eqref{:21a} holds for \eqref{:10r} and both \eqref{:10s} and \eqref{:10t}. 
\end{prop}
\begin{proof}
\As{J2}--\As{J5} except \eqref{:21T} are proved in \cite{o.rm,o.isde}. 
\eqref{:21T} is obvious for a Ginibre point process because 
their one-correlation functions with respect to the Lebesgue measure have 
a uniform bound such that $ \rho _{\mathrm{gin}}^{\nN , 1} \le 1/\pi $. 
This estimate follows from \thetag{6.4} in \cite{o.isde} immediately. 
Let $ \mathsf{d}_1$ and $ \mathsf{d}_2$ be the logarithmic derivative associated with ISDEs \eqref{:10s} and \eqref{:10t}. 
Then $ \mathsf{d}_1 = \mathsf{d}_2 $ a.s.\ \cite{o.isde}. Hence we conclude \pref{l:33} 
\end{proof}

\subsection{Gibbs measures with Ruelle-class potentials} \label{s:55}
Let $ \mu ^{\Psi }$ be Gibbs measures 
with Ruelle-class potential $ \Psi (x,y) = \Psi (x-y)$ 
that are smooth outside the origin. 
Let $ \Phi ^N \in C^{\infty}(\S )$ be a confining potential for the $ N $- particle system. 
We assume that the correlation functions of $ \mu ^{\Phi ^N ,\Psi }$ 
satisfy bounds 
$ \sup_N \rho ^{\nN , m } \le \cref{;35a} ^m $ for some constants 
$ \Ct{;35a}$; see the construction of \cite{ruelle.2}. 
Then one can see in the same fashion as \cite{o-t.tail} that 
$ \mu ^{\Psi }$ satisfy \As{J2}--\As{J5} except \eqref{:21T}. 
Under the condition $ \sup_N \rho ^{\nN , m } \le \cref{;35a} ^m $, 
\eqref{:21T} is obvious. 
Moreover, if $ \mu ^{\Psi }$ is a grand canonical Gibbs measure with 
sufficiently small inverse temperature $ \beta $, then $ \mu ^{\Psi }$ 
is tail trivial. Hence we can obtain \As{J6} in the same way as \cite{o-t.tail} in this case. 
We present two concrete examples below. 

\subsubsection{Lennard--Jones 6-12 potentials} \label{s:55a}
Let $\S =\R ^3$ and $\beta >0$. 
Let $\Psi _{6-12}(x)=|x|^{-12}-|x|^{-6}$ be the Lennard-Jones potential. 
The corresponding ISDEs are given by the following. 
\begin{align}\notag % \label{:95a}
dX_t^{N,i}=
& 
dB_t^i+\frac{\beta }{2} \{ 
\nabla \Phi ^N (X_t^{N,i}) + 
\sum_{j=1, \atop j\neq i}^{\nN }\frac{12(X_t^{N,i}-X_t^{N,j})}{|X_t^{N,i}-X_t^{N,j}|^{14}}- 
\frac{6(X_t^{N,i}-X_t^{N,j})}{|X_t^{N,i}-X_t^{N,j}|^{8}} \} dt 
\ (1\le i\le N)
,\\ \notag %\label{:95b} 
dX_t^i=
& 
dB_t^i+\frac{\beta }{2}\sum_{j=1,j\neq i}^\infty
\{ \frac{12(X_t^i-X_t^j)}{|X_t^i-X_t^j|^{14}}- \frac{6(X_t^i-X_t^j)}{|X_t^i-X_t^j|^{8}} \}dt\quad (i\in\N)
.\end{align}

\subsubsection{Riesz potentials} \label{s:56}
Let $d<a\in\N$ and $\beta >0$. Let $\Psi _a(x)=\frac{\beta }{a }|x|^{-a}$ the Riesz potential. 
%Let $ \muN $ be the associated grand canonical Gibbs measure. 
The corresponding SDEs are given by 
\begin{align}\notag %\label{:96a}
dX_t^{N,i}= & 
dB_t^i+\frac{\beta }{2}
 \{ 
\nabla \Phi ^N (X_t^{N,i}) + 
\sum_{j=1,j\neq i}^{\nN } \frac{X_t^{N,i}-X_t^{N,j}}{|X_t^{N,i}-X_t^{N,j}|^{2+a}}
\} 
dt \quad (1\le i\le N)
,\\\notag %\label{:96b}
dX_t^i= & 
dB_t^i+\frac{\beta }{2}\sum_{j=1,j\neq i}^\infty \frac{X_t^i-X_t^j}{|X_t^i-X_t^j|^{2+a}}dt \quad (i\in\N)
.\end{align}

\section{Acknowledgments} 

{}\quad 

\noindent 
Y.K. is supported by Grant-in-Aid for JSPS JSPS Research Fellowships (No.\!\! 15J03091).

\noindent 
H.O. is supported in part by a Grant-in-Aid for Scenic Research (KIBAN-A, No.\!\! 24244010; KIBAN-A, No.\!\! 16H02149; KIBAN-S, No.\!\! 16H06338) from the Japan Society for the Promotion of Science.

%%%%%%%%%%%%%%%%%%%%%%%%%%%%%%%%%%%%%%%%%%%%%%%%%%%%%%%%%%

\noindent 
Yosuke Kawamoto \\
Faculty of Mathematics, Kyushu University, Fukuoka 819-0395, Japan
\\
\texttt{y-kawamoto@math.kyushu-u.ac.jp}

\bs
\noindent 
Hirofumi Osada \\
Faculty of Mathematics, Kyushu University, Fukuoka 819-0395, Japan
\\
\texttt{osada@math.kyushu-u.ac.jp}
\\

\end{document}